\documentclass[a4paper,11pt]{amsart}

\usepackage{CJKutf8}

\usepackage{version}
\usepackage{amssymb}
\usepackage{amsfonts}
\usepackage{graphicx}
\usepackage{tikz}
\usetikzlibrary{arrows}

\usepackage{color}

\usepackage{epstopdf}
\usepackage[english]{babel}
\usepackage{float}
\usepackage{cite}
\usepackage{tikz,pgf}
\usetikzlibrary{patterns,spy}

\usepackage{geometry}
\usepackage[hidelinks]{hyperref}

\newtheorem{theorem}{Theorem}[section]

\newtheorem{lemma}[theorem]{Lemma}

\newtheorem{proposition}[theorem]{Proposition}

\newtheorem{conjecture}[theorem]{Conjecture}

\theoremstyle{definition}
\newtheorem{definition}[theorem]{Definition}

\newtheorem{remark}[theorem]{Remark}

\numberwithin{equation}{section}

\renewcommand\bigskip{\medskip}

\def\to{\rightarrow}
\def\wto{\rightharpoonup}

\def\D{\mathcal{D}}
\def\cP{\mathcal P}

\def\F{\mathcal F}
\def\N{\mathbb N}

\def\R{\mathbb R}

\def\Z{\mathbb Z}

\DeclareMathOperator{\sd}{{\rm{s}\textrm{-}\rm{dim}}}

\DeclareMathOperator{\dimH}{dim_H}
\DeclareMathOperator{\dimP}{dim_P}

\DeclareMathOperator{\spt}{spt}

\DeclareMathOperator{\ess-inf}{ess-inf}

\def\l{\left}
\def\r{\right}

\DeclareMathOperator{\dist}{{\rm dist}}

\begin{document}

\title[Projection theorems with countably many exceptions ]{Projection theorems with countably many exceptions and applications to the exact overlaps conjecture}

\author[Meng Wu]{Meng Wu}
\address{School of mathematics,  Hunan University,  Changsha,  410082, China; Department of Mathematical Sciences, P.O. Box 3000, 90014 University of Oulu, Finland}
\email{meng.wu@hnu.edu.cn}


\subjclass[2010]{}

\begin{abstract}

We establish several optimal estimates for exceptional parameters in the projection of fractal measures:
(1)	For a parametric family of self-similar measures satisfying a transversality condition, the set of parameters leading to a dimension drop is at most countable.
(2)	For any ergodic CP-distribution $Q$  on $\R^2$, the Hausdorff dimension of its orthogonal projection is $\min\{1, \dim Q\}$ in all but at most countably many directions.
Applications of our projection results  include:
(i) For any planar Borel probability measure with uniform entropy dimension $\alpha$, the packing dimension of its orthogonal projection is at least $\min\{1, \alpha\}$  in all but at most countably many directions.
(ii) For any planar set $F$, the Assouad dimension of its orthogonal projection is at least $ \min\{1, \dim_{\rm A} F\} $  in all but at most countably many directions.
\end{abstract}

\maketitle

\section{Introduction}
\subsection{Background and main results}
In the present paper,  we propose to resharpen and extend the classical  {\em Marstrand projection theorem} for a class of fractal measures possessing  certain geometric regularities.  Marstrand's projection theorem \cite{Marstrand1954} is a fundamental result in fractal geometry and geometric measure theory.  It states that for every Borel set $K$ in the plane,  the   Hausdorff dimension of the   orthogonal projection of  $K$ along a  generic direction  is preserved: for Lebesgue almost every $\pi\in G(2,1)$,
\begin{equation}\label{eq:1}
\dim_{\rm H} \pi  K = \min (1,\dim_{\rm H}K).
\end{equation}
Here $G(2,1)$ is the set of one-dimensional linear subspaces of $\R^2$,  $\pi $ stands for the orthogonal projection from $\mathbb{R}^2$ onto $\pi$  and $\dim_{\rm H}$ denotes the Hausdorff dimension.  This theorem has been quite influential,  its various variants or extensions often play key roles in many problems in fractal geometry,  geometric measure theory and dynamical systems (see e.g. \cite{FFJ,BHR,Shmerkin2015,HS2012,HS2015,BFLM}).   We note that there have been  recent major advances concerning sharpening of Marstrand's projection for general Borel sets (see e.g.  \cite{Orponen-Shmerkin-2023,Ren-Wang}). 

Let us call a direction $\pi\in G(2,1)$ exceptional  (for orthogonal projections of $K$)  if the equality  \eqref{eq:1} doesn't hold.  
Marstrand tells us that  the exceptional  directions form a set of zero Lebesgue measure.  Sometimes,  it is possible to give better estimates on the size of the exceptional set (in terms of Hausdorff dimension), but in general it is almost impossible to explicitly determine the exceptional directions.  In many applications of Marstrand's projection theorem,  however,  it is usually the estimate on the size of the exceptional set that really matters.   
In fact, there is a great desire to strengthen this classical result by providing classes of sets where there are no,  or very few,  exceptional directions.   While for general fractal sets,  explicitly determining the exceptional directions is intractable,  there is a widely accepted philosophy for sets with structures:

\medskip

{\em 
{\bf General Philosophy (rigidity phenomena).} For fractal sets with regular arithmetical or geometrical structures,  the exceptional set (of projections) should be very small and could only be caused by some evident algebraic or combinatorial reasons.
}
\medskip

In the world of fractal sets,    self-similar sets are among  the simplest and most fundamental class of objects, they possess very rigid  geometric structures.  A self-similar set has the property that it can be decomposed into parts which are exact replicas of the whole.  Specific to self-similar sets,   we have the following 
\begin{conjecture}\label{conj:proj-self-similar:1}
For any self-similar set  in the plane, the exceptional set  is at most
countable and explicitly determinable.
\end{conjecture}

When the self-similar set in question involves irrational rotations in its construction, Conjecture \ref{conj:proj-self-similar:1} was confirmed by Peres and Shmerkin \cite{Peres-Shmerkin}. In such a situation, there are no exceptional directions. The methods used to study projections of self-similar sets with rotations do not apply to self-similar sets without rotations.

For self-similar sets without rotations, Conjecture \ref{conj:proj-self-similar:1} is closely related to the study of self-similar sets on $\R$
with overlaps. This is because projections of such self-similar sets remain self-similar sets on 
$\R$,  often exhibiting overlaps in their construction. In this case, Conjecture \ref{conj:proj-self-similar:1} predicts that for a self-similar set without rotations, an exceptional projection can occur only if there exist two distinct cylinders in the construction that have exactly the same projection. Such exceptional projections form at most a countable set.
This conjecture is widely known as the {\em exact overlaps conjecture}, one of the central conjectures in fractal geometry. It has motivated many recent breakthroughs in the dimension theory of self-similar sets (see, e.g., \cite{Hochman-ICM,Varju-ICM}).

While Conjecture 1 (for self-similar sets without rotations) in its full generality is still open,  in recent years there have been some spectacular progress towards this conjecture (see e.g.  \cite{BV2019,Hochman2014, Hochman-MAMS,Rapaport-Ann.ENS,Shmerkin2014,Shmerkin2019,Varju-Annals,Rapaport-Varju}),  starting from the breakthrough of Hochman \cite{Hochman2014} who introduced the additive combinatorial methods  into the study of self-similar sets which has been revolutionary in the domain.  In \cite{Hochman2014},  it is shown that for every self-similar set in the plane, the set of exceptional directions is of Hausdorff (and Packing) dimension zero.  There are subsequent works following Hochman's,  many important special cases of Conjecture 1 have since been  proved to be true (see \cite{Hochman2014,  BV2019, Shmerkin2019,  Varju-Annals,Rapaport-Ann.ENS,  Rapaport-Varju,  FengFeng}).

Following the work of Hochman and the subsequent developments,  it is natural to ask whether additive combinatorial methods can be used to show that, for any self-similar set, the exceptional set is at most countable. In this paper, we confirm that this is indeed the case.  We prove a more general result for parametric family of self-similar measures satisfying a tranversality condition.  

Let us first recall the relevant notion.  
Let $\F=\{f_i(x)=r_ix+t_i\}_{i\in \Lambda}$ be a family of contracting affine maps on $\R$.  Such a family $\F$ is called an iterated function system (IFS).  Given  a probability vector $p=(p_i)_{i\in \Lambda}$,  there exists a unique probability measure $\mu_p$ on $\R$ satisfying 
\[\mu_p=\sum_{i\in \Lambda}p_if_i\mu_p.\]
The measure $\mu_p$ is called the self-similar measure associated to $\F$ and $p$. The so called similarity dimension of $\mu_p$ is  given by $\sd \mu_p=\frac{\sum_ip_i\log p_i}{\sum_ip_i\log |r_i|}$. 
\begin{theorem}\label{thm:main:-2}
Let $J$ be an open  interval, and for each $t\in J$, let $\F_t$ be an IFS on $\R$. For a given probability vector $p=(p_i)_{i\in \Lambda}$, let $\mu_{p,t}$ be the self-similar measure associated to $\F_t$ and $p$.  Suppose that the family $\{\F_t\}_{t\in J}$ satisfies the $\beta$-transversality condition.  Then the following set is at most countable:
\[
\{t\in J: \dim \mu_{p,t}<\min (1,\sd \mu_{p,t})\}.
\]
\end{theorem}

 We refer to Section \ref{section: result-projection-measure-self-similar-IFS:1} for details on the notion of the $\beta$-transversality condition and examples of IFSs that satisfy it. 
\begin{remark}
\begin{itemize}
\item[(1)] The family of orthogonal projections of any self-similar measure without rotations satisfies the $\beta$-transversality condition. Thus, Theorem \ref{thm:main:-2} implies that the exceptional set for projections of any such measure is at most countable.
\item[(2)]  It seems plausible that the exceptional estimates in Theorem \ref{thm:main:-2} also hold for certain nonlinear IFSs on $\R$. For instance, for a one-parameter family of Furstenberg measures (as studied in \cite[Section 6.1]{HochmanSolomyak}) satisfying an analogous transversality condition to the one imposed in Theorem \ref{thm:main:-2}, it may be possible to combine arguments from \cite{HochmanSolomyak} with those in the present paper to show that the exceptional set for the dimension of Furstenberg measures is at most countable. However, we shall not pursue this further here.
\end{itemize}
\end{remark}

 In fact, we prove the following stronger result, from which Theorem \ref{thm:main:-2} follows as an immediate corollary. This theorem is of independent interest.
 \begin{theorem}\label{thm:main:pertubate-increase:1}
Let $J, \{\F_t\}$ and $\mu_{p,t}$ be as in Theorem \ref{thm:main:-2}.  Suppose that the family $\{\F_t\}_{t\in J}$ satisfies the the $\beta$-transversality condition.   Then for any $t\in J$ with $\dim \mu_{p,t}<\min (1,\sd \mu_{p,t})$,  there exist $\delta>0$ and $\epsilon>0$ such that 
\[
 \dim \mu_{p,t'} \ge \dim \mu_{p,t}+\delta  \  \  \  \textrm{ for } t'\in \left([t-\epsilon,t+\epsilon]\cap J\right)\setminus \{t\}. 
\]
\end{theorem}
 
Let us show how to deduce Theorem \ref{thm:main:-2} from Theorem \ref{thm:main:pertubate-increase:1}.  
\begin{proof}[Proof of Theorem \ref{thm:main:-2}]
Let 
\[
E=\{t\in J: \dim \mu_{p,t}<\min (1,\sd \mu_{p,t})\}.
\]
It is easy to check that $E$ is a Borel set.  Suppose that $E$ is uncountable.  Then there exists a non-atomic Borel probability measure $\nu$ supported on a compact subset of $E$.  By Theorem \ref{thm:main:pertubate-increase:1}, for each $t\in E$, 
there exist $\delta(t)>0$ and $\epsilon(t)>0$ such that for all $t'\in \left(B(t,\epsilon(t))\cap J\right)\setminus \{t\}$, we have 
\begin{equation}\label{eq:proof-Theorem-self-similar-transversality:100}
\dim \mu_{p,t'} \ge \dim \mu_{p,t}+\delta(t).
\end{equation}
By Egorov's theorem,  there exist a Borel set $E_1\subset E$ with $\nu(E_1)>0$ and a constant $\rho>0$ such that $\delta(t)\ge \rho$ and $\epsilon(t)\ge \rho$ for each  $t\in E_1$.  Since $\nu$ is non-atomic,  the subset $E_1$ must be uncountable and in particular,  infinite.  Thus there exist $t_1,t_2\in E_1$ such that $0<|t_1-t_2|<\rho$.  By \eqref{eq:proof-Theorem-self-similar-transversality:100}, we must have
\[
\dim \mu_{p,t_1}\ge \dim \mu_{p,t_2}+\rho \ \textrm{ and } \dim \mu_{p,t_2}\ge \dim \mu_{p,t_1}+\rho.
\]
This is  clearly impossible.
\end{proof}

Our methods also enable us to establish sharp projection theorems for a broader class of measures that satisfy only a statistical self-similarity property. These measures, known as CP-chain measures, were introduced by Furstenberg \cite{Furstenberg69} and later significantly developed by Hochman and Shmerkin \cite{HS2012} as a key tool for studying the local structure of measures, particularly projections of fractal measures.
Briefly, given a finite measure $\mu$ on $\R^d$ and a point $x\in {\rm supp} (\mu)$, we consider the sequence of measures $\mu^{\D_n(x)}$, obtained by restricting $\mu$ to the dyadic cell $\D_n(x)$ containing $x$, then normalizing and rescaling it to the unit cube. The essential requirement on $\mu$ is that, for $\mu$-typical $x$, the sequence $(\mu^{\D_n(x)})_n$ is generic for some distribution $Q$ on probability measures. This limiting distribution $Q$ is called a CP-distribution, and its dimension, denoted $\dim Q$, is defined as the average dimension of measures with respect to $Q$.
For further details on CP-distributions, see Section \ref{subsection: properties-of-CP-dist:1}.

\begin{theorem}\label{thm:proj-thm-cp-distribution:-2}
Let $Q$ be an ergodic CP-distribution on $\R^2$ with respect to the dyadic partition.  Then the following set is at most countable:
\[
\{\pi\in G(2,1): \dim \pi_\theta(Q)<\min(1,\dim Q)\}.
\]
\end{theorem}

Thanks to the work of Furstenberg, Hochman, and Shmerkin on CP-distributions and their connection to the local structure of fractal measures, results on CP-distribution projections can be directly applied to derive projection theorems for general measures with certain regularity properties. In particular, this allows Theorem \ref{thm:proj-thm-cp-distribution:-2} to be used in the study of projections of regular fractal measures.  Below, we list some of these applications. For further details and additional applications, see Section \ref{section: precise-statements-results:0}.

Our first application concerns projections of measures satisfying a regularity condition known as {\em uniform entropy dimension}, introduced in \cite{Hochman2014}.  Roughly speaking, a measure $\mu$ has uniform entropy dimension $\alpha$ if, for $\mu$-most $x \in \text{supp}(\mu)$  and most $n \in \mathbb{N}$, the entropies of the component measures $\mu^{\mathcal{D}_n(x)}$ concentrate around $\alpha$.
For more details, see Section \ref{section: meas-uniform-ent-dimension:0}. We denote the packing dimension by $\dimP$.

\begin{theorem}\label{thm:main:-1}
Let $\mu$ be a finite Borel measure on $\R^2$.  Suppose that $\mu$ has upper uniform entropy dimension $\alpha$.  Then 
\[
\dim_{\rm P}\pi \mu=\min(1,\alpha)
\]
for all $\pi\in G(2,1)\setminus E$ where $E$ is at most countable.
\end{theorem}

Many measures in fractal geometry and dynamical systems satisfy the assumption of uniform entropy dimensionality. These include, for example, all self-similar and self-conformal measures, many self-affine measures, all Ahlfors-David regular measures, and numerous dynamically defined Cantor measures (such as $\times p$-invariant ergodic measures on the torus).

We note that under the assumption of Theorem \ref{thm:main:-1},  our conclusion only concerns the Packing dimensions of the projected measures.  Indeed, one cannot strengthen our result by replacing the packing dimension with the Hausdorff dimension in the conclusion—there exists a measure with positive uniform entropy dimension but zero Hausdorff dimension.
It is worth pointing out that in many situations when our theorem is applicable, the projected measures usually have equal Hausdorff and Packing dimensions.    In such situation, our results do give information on Hausdorff dimension of projections.
As with the case of previous variants of Marstrand's projection theorem,  we anticipate that our results will have applications for various problems.  

Our second application of Theorem \ref{thm:proj-thm-cp-distribution:-2} is a sharp projection theorem for Assouad dimension of general planar sets.  We use $\dim_{\rm A}F$ to denote the Assouad dimension of a set $F$.  

\begin{theorem}\label{thm:proj-Assouad:-2}
Let $F \subset \R^2$ be any non-empty set. Then the following set is at most countable
\[
\{\pi\in G(2,1): \dim_{\rm A}\pi  F<\min(1,\dim_{\rm A}F)\}.
\]
\end{theorem}
For more details about Assouad dimension,  we refer to Section \ref{section: assouad-projection:0}.

\begin{remark}
While preparing this preprint for publication, we learned that T. Orponen \cite{Orponen-2024} has obtained related results on Theorem \ref{thm:main:-1}. As a corollary of his results, Orponen showed that if the measure $\mu$ is Ahlfors-regular, then the exceptional set in Theorem \ref{thm:main:-1} has Hausdorff dimension zero. From this perspective, Orponen’s results appear weaker than ours; however, he also provides a discretized and quantitative version \cite[Theorem 1.3]{Orponen-2024}, which seems to be of independent interest and cannot be obtained using our methods.  Orponen’s approach differs from the one used in this paper.  
\end{remark}

\subsection{Precise statements of results}\label{section: precise-statements-results:0}
\subsubsection{Self-similar sets and measures}\label{section: result-projection-measure-self-similar-IFS:1}

Let $J\subset \R$ be an open interval. For each $t\in J$, let $\F_t=\{f_{i,t}(x)=r_{i,t}x+s_{i,t}\}_{i\in \Lambda}$ be an affine IFS on $\R$.  We shall use the metric $d$ on $\Lambda^\N$ defined by $d(x,y)=2^{-|x\wedge y|}$,  where $x\wedge y$ is the longest common initial segment of $x$ and $y$ and $|x\wedge y|$ is the length of the segment (i.e.,  $|x\wedge y|=\min(k:x_k\neq y_k)-1$).  For $x,y\in \Lambda^\N$, let $\Delta_{x,y}(t)=f_{x,t}(0)-f_{y,t}(0)$,  where $f_{x,t}(0)=\lim_{n\to \infty} f_{x_1,t}\circ\ldots \circ f_{x_n,t}(0)$.    

We always assume that for all $i\in \Lambda$, the maps $t\mapsto r_{i,t}$ and  $t\mapsto s_{i,t}$ are $C^2$ on $J$ and  for any compact $J'\subset J$  there exist $C>0$ such that 
\begin{equation}\label{eq:definition-IFS-regularity:1}
\left|\frac{d^k}{dt^k}f_{x,t}(0)\right| \le C 
\end{equation}
for all $x\in \Lambda^\N,  t\in J'$ and $k=1,2$.
\begin{definition}\label{definition:transversality:1}
The family $\{\F_t\}_{t\in J}$ satisfies the $\beta$-transversality condition on $J$ if there exist constants $\beta\ge 0$ and  $C_\beta$ such that for any $t\in J$ and $x,y\in \Lambda^\N$,  the condition $|\Delta_{x,y}(t)|\le C_\beta d(x,y)^\beta$ implies 
\begin{equation}\label{eq:definition-tranversal:1}
|\Delta_{x,y}'(t)|\ge C_\beta d(x,y)^\beta.  
\end{equation}
\end{definition}

This notion of transversality was introduced in \cite{Peres-Schlag}. In the literature, it is often assumed that $\beta\le 1$,  however, we do not impose this requirement in this paper. Clearly, the larger $\beta$  is, the weaker the transversality condition becomes.

Let us restate our result on parametric family of self-similar measures (Theorem \ref{thm:main:-2})  
\begin{theorem}\label{thm:proj-parametric-self-similar:1}
Let $J$ be an open  interval, and for each $t\in J$, let $\F_t$ be an IFS on $\R$. For a given probability vector $p=(p_i)_{i\in \Lambda}$, let $\mu_{p,t}$ be the self-similar measure associated to $\F_t$ and $p$.  Suppose that the family $\{\F_t\}_{t\in J}$ satisfies the $\beta$-transversality condition.  Then the following set is at most countable:
\[
\{t\in J: \dim \mu_{p,t}<\min (1,\sd \mu_{p,t})\}.
\]
\end{theorem}

Let us present a class of IFSs for which the transversality condition is (partially) satisfied.  Let $n\ge 2$ and $t_1,\ldots,t_n\in \R$. For $\lambda\in (0,1)$,  consider the IFS $\F_\lambda=\{f_{i,\lambda}(x)=\lambda x+t_i\}_{i=1}^n$.  Given probability vector $p=(p_i)_{i=1}^n$, let $\nu_{p,\lambda}$ be the self-similar measure associated to $\F_\lambda$ and $p$.

When $n=2$ and $p=(\frac{1}{2},\frac{1}{2})$, the measures $\nu_{p,\lambda}$ correspond to the classical Bernoulli convolutions.  Varj\'u \cite{Varju-Annals} proved the exact overlaps conjecture in this case, implying that the exceptional parameters $\lambda$ such that $\dim \nu_\lambda<\min(1,\sd \nu_\lambda)$ form at most a countable set and must satisfy certain algebraic equations (i.e., there exist $x_1^n\neq y_1^n\in \Lambda^n$ such that $f_{x_1^n,\lambda}(0)=f_{y_1^n,\lambda}(0)$).   The proof of Varj\'u combines several deep results on self-similar measures,   which consist of hard analysis.  By applying Theorem \ref{thm:proj-parametric-self-similar:1} and leveraging existing transversality results for Bernoulli convolutions \cite{Solomyak-Annals, Peres-Solomyak-simple}, we provide a soft alternative proof that the set of exceptional parameters remains at most countable.

We now suppose $n\ge 3$.  Let $b=\frac{\max_{i\neq j}|t_i-t_j|}{\min_{i\neq j}|t_i-t_j|}$.  By \cite[Corollary 5.2]{Peres-Solomyak-1998}  and \cite[Lemma 5.9]{Peres-Schlag},  for any $\epsilon>0$,  there exist $\beta>0$ such that the IFSs $\{\F_{\lambda}\}_\lambda$ satisfy the $\beta$-transversality condition on the interval $[\epsilon,(1+\sqrt{b})^{-1}]$.  It follows that, by Theorem \ref{thm:proj-parametric-self-similar:1}, the following set is at most countable:
\[
\left\{\lambda\in \left(0,(1+\sqrt{b})^{-1}\right): \dim \nu_{p, \lambda}<\min(1,\sd \nu_{p, \lambda})\right\}.
\]

\subsubsection{Measures with uniform entropy dimension}\label{section: meas-uniform-ent-dimension:0}

\begin{definition}\label{def:upper-unif-entrop-dim:1}
Let $\mu\in \cP(\R^d)$. The {\em upper uniform entropy dimension} of $\mu$ is defined to be the supremum of $\alpha\ge 0$ such that 
\begin{equation}\label{eq:def:upper-unif-entrop-dim:1}
\liminf_{\epsilon\to0}\limsup_{l\to\infty}\limsup_{n\to\infty}\mu\left( x:\left| \left\{ 1\le k\le n: \left|H(\mu^{\D_k(x)},\D_l)-l\alpha\right| \le l \epsilon \right\} \right| \ge n(1-\epsilon) \right) =1.
\end{equation} 
\end{definition}
Here $H(\eta,\mathcal{A})$ denotes the Shannon entropy of a measure $\eta$ with respect to a partition $\mathcal{A}$.
Thus,  if $\mu$ has  upper uniform entropy dimension $\alpha$, then for any $\epsilon>0$,  there exist $l\in \N$ such that 
\begin{equation}\label{eq:def:upper-unif-entrop-dim:2}
\limsup_{n\to\infty}\mu\left( x:\left| \left\{ 1\le k\le n: \left|H(\mu^{\D_k(x)},\D_l)-l\alpha\right| \le l \epsilon \right\} \right| \ge n(1-\epsilon) \right) \ge 1-\epsilon.
\end{equation}

As we mentioned above, many measures with dynamical origins or regular geometric structures are uniform entropy dimensional. 
\begin{remark}\label{remark:1}
\begin{itemize}
\item[(1)]  Uniform entropy dimension first appeared in work of Hochman \cite{Hochman2014}, a measure $\mu\in \cP(\R^d)$ has uniform entropy dimension $\alpha$ of for any $\epsilon>0$ exists $l\in \N$ such that 
\[
\liminf_{n\to\infty}\mu\left( x: \left| \left\{ 1\le k\le n: \left|H(\mu^{\D_k(x)},\D_l)-l\alpha\right| \le l \epsilon \right\} \right| \ge n(1-\epsilon) \right) \ge 1-\epsilon.
\]
It is clear that if a measure has uniform entropy dimension $\alpha$, then it has upper uniform entropy dimension at least $\alpha.$
\item[(2)] Recall that the Packing dimension of measure $\eta\in \cP(\R^d)$ is defined as
\[
\dimP \eta=\inf\{ \dimP A: \eta(A)>0 \}.
\]
 If a measure has upper uniform entropy dimension $\alpha$,   then its packing dimension is at least $\alpha$.  The  upper uniform entropy dimension of a measure doesn't tell anything about  the Hausdorff dimension of the measure: for any $\alpha\ge 0$,  there exists measure whose upper uniform entropy dimension is $\alpha$ but whose Hausdorff dimension is zero.
\end{itemize}
\end{remark}

\begin{theorem}\label{thm:main-proj-upper-unif-entropy-dim}
Let $\mu\in \cP(\R^2)$. Suppose that $\mu$ has upper uniform entropy dimension $\alpha$.
Then the following set is at most countable
\begin{equation}\label{eq:thm:main-proj-upper-unif-entropy-dim:1}
\{\pi\in G(2,1): \dimP \pi\mu <\min (1,\alpha)\}.
\end{equation}
\end{theorem}

\begin{remark}
As mentioned in Remark \ref{remark:1},  a measure $\mu$ with  upper uniform entropy dimension $\alpha>0$ could have zero Hausdorff dimension,  hence in \eqref{eq:thm:main-proj-upper-unif-entropy-dim:1}, one cannot replace $\dimP \pi\mu$ by $\dimH \pi\mu$. 
\end{remark}

\subsubsection{Assouad dimension of projections}\label{section: assouad-projection:0}
As another corollary of Theorem \ref{thm:proj-thm-cp-distribution:-2}, we have the following sharp projection theorem for Assouad dimension of sets.  We use $\dim_{\rm A}F$ to denote the Assouad dimension of a set $F$.  

\begin{theorem}\label{thm:proj-Assouad:-1}
Let $F \subset \R^2$ be any non-empty set. Then the following set is at most countable
\[
\{\pi\in G(2,1): \dim_{\rm A}\pi F<\min(1,\dim_{\rm A}F)\}.
\]
\end{theorem}
For more details about Assouad dimension and recent developments around the notion,    we refer to \cite{Fraser-book}.  For any subset $F \subset \R^2$,  that the exceptional set $\{\pi\in G(2,1): \dim_{\rm A}\pi F<\min(1,\dim_{\rm A}F)\}$ has zero Hausdorff dimension is due to Orponen \cite{Orponen-Assouad}. The question as to whether the set $E_{\rm A}(K)$ also has zero Packing dimension has been asked in  \cite{Orponen-Assouad, Fraser-book}. 
Theorem \ref{thm:proj-Assouad:-1} is sharp: there exists a self-similar set $F\subset \R^2$ satisfying the open set condition (thus $F$ is Alfhors-David regular) such that $\{\pi\in G(2,1): \dim_{\rm A}\pi F<\min(1,\dim_{\rm A}F)\}$  is dense in $G(2,1)$ (hence at least countable).  See Section \ref{subsection:Proj-Assouad}.

\subsubsection{Organization of the paper}
The remaining parts of the paper are organized as follows. In the next section, we introduce various notations used in the paper. In Section \ref{section:easy-proof-proj-self-similar-without-rotation}, we present the proof of Theorem \ref{thm:main:pertubate-increase:1}. 
The proof of Theorem \ref{thm:main:pertubate-increase:1} will also serve as a motivating example for the proof of Theorem \ref{thm:proj-thm-cp-distribution:-2} which follows a similar strategy but technically more complicated.
In Section \ref{section: proj-CP-distributions:1}, we prove Theorem \ref{thm:proj-thm-cp-distribution:-2}. In Section 3, we present the proofs of Theorems \ref{thm:main-proj-upper-unif-entropy-dim} and \ref{thm:proj-Assouad:-1}.

\section{Notations}\label{section:notation}
For a measure $\eta$ and real $r\neq 0$, $r\eta$ stands for the measure defined
by $r\eta(A)=\eta(A/r)$. 
Recall that $\mu*\nu$ stands for the convolution of two measures $\mu$ and $\nu$.  

For two real vectors $(a_i)_{i=1}^m,   (b_i)_{i=1}^m$,  $\dist((a_i)_i,(b_i)_i)$ denotes the Euclidean distance between them,  i.e.,  
$\dist((a_i)_i,(b_i)_i)=\left(\sum_i|a_i-b_i|^2\right)^{1/2}$. 

The collection of Borel probability measures on a metric space $X$ is denoted by $\mathcal{P}(X)$.  For $\mu,\nu\in \mathcal{P}(X)$, we use the L\'evy-Prokhorov metric to measure their distance:
$$d(\mu,\nu)=\inf \{\epsilon>0: \mu(A)\le \nu(A^{\epsilon})+\epsilon \ \textrm{ for all Borel set }A\},$$
where $A^{\epsilon}$ is the $\epsilon$-neighborhood of $A$ in $X$.   

For a measure $\mu$ and a subset $A$ with $0<\mu(A)<\infty$, we write  $\mu_A=\frac{\mu_{|A}}{\mu(A)}$, where $\mu_{|A}$ is the restriction of $\mu$ on $A$.

For integers $d\ge 1, k\ge 0$,   we use  $\mathcal{D}_k(\R^d)$ to denote the collection of dyadic cells of side-length $2^{-k}$:
\[
\mathcal{D}_k(\R^d)=\left\{ [\frac{l_1}{2^k},\frac{l_1+1}{2^k})\times \ldots \times [\frac{l_d}{2^k},\frac{l_d+1}{2^k}) : l_1,\ldots, l_d\in \Z \right\}.
\]
For $x\in \R^d$,  $\D_k(x)$ is the unique element of $\mathcal{D}_k(\R^d)$ containing $x$.

For a measure $\mu$ and  $D\in \D_k(\R^d)$, we denote 
\[
\mu^D=S_D(\mu_D),
\]
where  $S_{D}$ is the unique orientation-preserving homothety sending $D$ to $[0,1)^d$. 

For $\mu\in  \mathcal{P}(X)$ and a countable partition $\mathcal{A}$ of $X$,  the Shannon entropy of $\mu$ with respect to $\mathcal{A}$ is given by 
\[H(\mu,\mathcal{A})=-\sum_{A\in \mathcal{A}} \mu(E)\log \mu(E),
\]
where the logarithm is in base $2$ and $0\log 0 = 0$.  

The (lower) Hausdorff dimension of a measure $\mu\in \cP(\R^d)$, denoted $\dim_{\rm H} \nu$, is defined as 
\begin{equation}\label{def-low-Hausdorff-dim-1}
\dim_{\rm H} \mu=\inf\{\dimH A: A\subset \R^d \textrm{ is Borel and } \mu(A)>0\}.
\end{equation}
where $\dimH A$ denotes the Hausdorff dimension of $A$.
The (lower) Packing dimension of a measure $\mu\in \cP(\R^d)$, denoted $\dim_{\rm P} \mu$, is defined as 
\begin{equation}\label{def-low-Pack-dim-1}
\dim_{\rm P} \mu=\inf\{\dimP A: A\subset \R^d \textrm{ is Borel and }  \mu(A)>0\}.
\end{equation}
It is well known that (see e.g.  \cite{FanLauRao}) the dimensions $\dimH \mu$ and $\dimP\mu$ can be characterized alternatively as follows
\begin{equation}\label{eq:chara-hausdorff-dim-meas:1}
\dimH \mu = \ess-inf_{x\sim \nu} \underline{D}(\mu,x),  
\end{equation}
\begin{equation}\label{eq:chara-packing-dim-meas:1}
\dimP \mu = \ess-inf_{x\sim \mu} \overline{D}(\mu,x),  
\end{equation}
where $ \underline{D}(\mu,x)=\liminf_{n\to\infty}\frac{\log \mu(\D_n(x))}{-n\log 2}$ is the lower local dimension of $\mu$ at $x$ and similarly $ \overline{D}(\mu,x)$ is the upper local dimension of $\mu$ at $x$.
When $\underline{D}(\mu,x)= \overline{D}(\mu,x)$, we say that the local dimension of $\mu$ at $x$ exists and denote
it by $D(\mu,x)$.  If the local dimension of $\mu$ exists and is constant $\mu$-almost everywhere,  then
$\mu$ is called exact dimensional and the almost sure local dimension is denoted by $\dim \mu$.

For $A\subset\R^d$ and $\delta>0$, we denote by $N_{\delta}(A)$ the smallest number of balls of diameter $\delta$ needed to cover $A$.

We use the small $o$ notation: $o_c(1)$ denotes a small quantity depending on a parameter constant $c$ with $o_c(1)\to 0$ as $c\to 0$ or $c\to\infty$. Whether $c\to 0$ or $c\to\infty$ will be clear from the context.

The notation $A=B\pm C$, with $C\ge 0$,  means $B-C\le A\le B+C$.

Let $\mu\in \cP(\R^2)$ and $\pi\in G(2,1)$.  For $x\in \R$,  let $\mu_{\pi^{-1}(x)}$ denote the conditional measure of $\mu$ on the fiber $\pi^{-1}(x)$,  provided that $\mu_{\pi^{-1}(x)}$ is well defined.  For $\pi\mu$-a.e.  $x$, the conditional measure $\mu_{\pi^{-1}(x)}$ is well defined and for any Borel set $E\subset \R^2$,  we have 
\begin{equation}\label{eq:notation-conditional-meas:1}
\mu(E)=\int  \mu_{\pi^{-1}(x)}(E) d\pi\mu(x).
\end{equation}

\section{Self-similar measures}\label{section:easy-proof-proj-self-similar-without-rotation}
In this section, we prove  Theorem  \ref{thm:main:pertubate-increase:1}.  
\begin{theorem}\label{thm:main:pertubate-increase:2}
Let $J$ be an open  interval, and for each $t\in J$, let $\F_t$ be an IFS on $\R$. For a given probability vector $p=(p_i)_{i\in \Lambda}$, let $\mu_{p,t}$ be the self-similar measure associated to $\F_t$ and $p$.  Suppose that the family $\{\F_t\}_{t\in J}$ satisfies the $\beta$-transversality condition.  Then for any $t\in J$ with $\dim \mu_{p,t}<\min (1,\sd \mu_{p,t})$,  there exist $\delta>0$ and $\epsilon>0$ such that 
\[
 \dim \mu_{p,t'} \ge \dim \mu_{p,t}+\delta  \  \  \  \textrm{ for } t'\in \left(B(t,\epsilon)\cap J\right)\setminus \{t\}. 
\]
\end{theorem}

The main ingredient in the proof of Theorem \ref{thm:main:pertubate-increase:2} is Hochman's inverse theorem for entropy. Let us recall this result. 
\begin{theorem}[Theorem 4.11 of \cite{Hochman2014}]\label{theorem:Hochman-inverse-thm}
For  every $1\le k_1,k_2 \in \N, \epsilon>0$,   there exists $\delta>0$ such that the following holds for all $n$ large enough. Let $\mu,\nu\in \cP([0,1])$. Suppose that 
\[H(\mu*\nu,\D_{n})\le H(\mu,\D_{n})+n\delta.\]
Then there exist $I,J\subset \{1,\ldots,n\}$ such that $|I\cup J|\ge (1-\epsilon)n$ and 
\begin{eqnarray*}
\mu\left(x: H(\mu^{\D_{i}(x)},\D_{k_1})\ge (1-\epsilon)k_1 \right) & > & 1-\epsilon \  \textrm{ for } i\in I  \\
\nu\left(x: \nu^{\D_{j}(x)}(Q)\ge (1-\epsilon) \textrm{ for some } Q\in \D_{k_2} \right) & > & 1-\epsilon \  \textrm{ for } j\in J. 
\end{eqnarray*}
\end{theorem}

We shall use the following standard facts concerning the entropy function $H(\mu,\D_k)$. 
\begin{lemma}\label{lemma:almost-continuity-entropy:1}
Let $\eta\in \cP(X)$, where $X$ is a metric space.   Suppose that  $\mathcal{A}_1,\mathcal{A}_2$ are partitions of $X$,  and $k\in \N$.
\begin{itemize}
\item[(1)] If each element of $\mathcal{A}_1$ intersects at most $k$ elements of $\mathcal{A}_2$ and vice versa, then 
$\left| H(\eta,\mathcal{A}_1)-H(\eta,\mathcal{A}_2) \right|=O(\log k)$.
\item[(2)] If  $X=\R^d$ and  $h(x)=rx+s$ with $r\in \R, s\in \R^d, 1/2\le |r|\le 2$,  then $\left| H(\eta,\D_k)-H(h\eta,\D_k) \right|=O_d(1)$.
\item[(3)] If $\mathcal{A}_2$ refines $\mathcal{A}_1$,  then 
\[H(\eta,\mathcal{A}_2) = H(\eta,\mathcal{A}_1)+\sum_{a_1\in \mathcal{A}_1}\eta(a_1)\sum_{a_2\in \mathcal{A}_2}\frac{\eta(a_1\cap a_2)}{\eta(a_1)}\log  \frac{\eta(a_1)}{\eta(a_1\cap a_2)}.  \] 
\end{itemize}
\end{lemma}

\begin{lemma}\label{lemma:convexity-entropy:1}
If $\eta=\sum_{i}q_i\eta_i,  \eta_i\in \cP(\R^d),\sum_iq_i=1, q_i\ge 0$, then for any partition $\mathcal{A}$ of $\R^d$,  we have
\begin{itemize}
\item[(1)]
$
H(\eta,\mathcal{A})\ge \sum_i q_i H(\eta_i,\mathcal{A});
$
\item[(2)]
$
H(\eta,\mathcal{A})\le \sum_{i}q_i\log \frac{1}{q_i}+ \sum_i q_i H(\eta_i,\mathcal{A}).
$
\end{itemize}
\end{lemma}

Let us recall some standard notations.   Let $\Lambda^*=\cup_{k\ge 1}\Lambda^n$.   For $I=i_1\ldots i_k\in \Lambda^*$,  denote $p_I=p_{i_1}\cdot p_{i_2}\cdot  \ldots \cdot p_{i_k}$,  $f_{I,t}(x)=f_{i_1,t}\circ \ldots \circ f_{i_k,t}(x)=r_{I,t}x+s_{I,t}.$ The natural coding map $\pi_t$ associated to the IFS $\F_t$ is the map from $\Lambda^\N$ to the attractor of $\F_t$ defined by 
\[
\pi_t(x)=f_{x,t}(0)=\lim_{n\to \infty} f_{x_1^n,t}(0).
\]

\begin{proof}[Proof of Theorem  \ref{thm:main:pertubate-increase:2}]
We fix a probability vector $p=(p_i)_{i\in \Lambda}$ and in the following simply write $\mu_t$ for $\mu_{p,t}$.  We assume that $p_i>0$ for each $i\in \Lambda$.  

Let us fix $t_0\in J$ and suppose that 
\[
\dim \mu_{t_0} <\min (1,  \sd \mu_{t_0}).
\]
Since self-similar measures are  exact dimensional (see \cite{FengHu}),   we have 
$\lim_{n\to \infty} \frac{1}{n}H(\mu_{t_0},\D_n)=\dim \mu_{t_0}.$  Hence  for any $\epsilon>0$,  there exist $n_0\in \N$ such that
\begin{equation}\label{eq:pf:thm-pertubate-increase-2:1}
\left| \frac{1}{n}H(\mu_{t_0},\D_n)-\dim \mu_{t_0} \right| < \epsilon \ \textrm{ for } n\ge n_0. 
\end{equation}
By our assumptions, there exists $\beta\ge 0$ such that the IFSs $\{\F_t\}_{t\in J}$ satisfy the $\beta$-transversality condition (Definition \ref{definition:transversality:1}).    Let $m=\lfloor \beta+3 \rfloor$.  Here $\lfloor s \rfloor$ stands for the integer part of a real number $s$.

For $k\in \N$,   let $\Lambda_k=\{ I\in \Lambda^*: |r_{I,t_0}| \le 2^{-k}<|r_{I^-,t_0}|\}$,   where $I^-=i_1\ldots i_{l-1}$ if $I=i_1\ldots i_{l}$.  
Then $\{[I]: I\in \Lambda_k\}$ forms a partition of $\Lambda^\N$.   
Note that we have \[\lim_{k\to\infty}\frac{1}{k}\sum_{I\in \Lambda_k} p_I\log \frac{1}{p_I}= \sd \mu_{t_0}.\] Thus there exist $n_1\in \N$ such that 
\begin{equation}\label{eq:pf:thm-pertubate-increase-2:2}
\left| \frac{1}{k}\sum_{I\in \Lambda_k} p_I\log \frac{1}{p_I} - \sd \mu_{t_0} \right| < \epsilon \ \textrm{ for } k\ge n_1. 
\end{equation}
Note that every self-similar measure has uniform entropy dimension which is equal to its dimension (\cite[Proposition 5.2]{Hochman2014}).  Thus there exists large $l\ge 1$ and  $n_2\in \N$ such that for $k\ge n_2$,
\begin{equation}\label{eq:pf:thm-pertubate-increase-2:2.1}
\mu_{t_0} \left( x:\left| \left\{ 1\le j\le k: \left|H(\mu_{t_0}^{\D_j(x)},\D_l)-l\dim\mu_{t_0}\right| \le l \epsilon \right\} \right| \ge k(1-\epsilon) \right) \ge 1-\epsilon.
\end{equation}

Let $\epsilon_0=2^{-(n_0+n_1+n_2) m}$.  
In the following,  we are going to show that there exists $\delta>0$ such that 
\begin{equation}\label{eq:pf:thm-pertubate-increase-2:3}
\dim \mu_{t} \ge \dim \mu_{t_0} +\delta \ \textrm{ for } t\in (B(t_0,\epsilon_0)\cap J)\setminus \{t_0\}.
\end{equation}
Let us fix any $t\in (B(t_0,\epsilon_0)\cap J)\setminus \{t_0\}$.   We also fix $n\in \N$ such that 
\[|t-t_0| = 2^{-nm\pm m}.\]
Note that,  necessarily,  we have \[n\ge n_0+n_1+n_2-1\ge \max(n_0,n_2,n_2).\]
Let us denote the attractor of $\F_{t_0}$ by $K$.    We fix any  $x_0\in K$.
For $Q\in \D_n(\R)$, let 
\[
A_Q=\left\{ I\in \Lambda_n: f_{I,t_0}(x_0)\in Q \right\}.
\]
For $Q\in \D_n(\R)$ with $A_Q\neq \emptyset$, let 
\[
a_Q=\sum_{I\in A_Q}p_I,   \  \  \nu_Q^{t_0}=\frac{1}{a_Q}\sum_{I\in A_Q}f_{I,t_0}\mu_{t_0}\  \textrm{ and }  \nu_Q^{t}=\frac{1}{a_Q}\sum_{I\in A_Q}f_{I,t}\mu_{t}.
\]
Note that for  $I\in \Lambda_n$,  $f_{I,t_0}$ has contraction ratio $\le 2^{-n}$,  thus  $f_{I,t_0}(K)$ has diameter at most $2^{-n}$ times the diameter of $K$.   It follows that $f_{I,t_0}(K)$ intersects at most $O_{K}(1)$ elements of $\D_{n}(\R)$.
By Lemma \ref{lemma:almost-continuity-entropy:1}  (1),  we have
\begin{equation}\label{eq:pf:thm-pertubate-increase-2:3.1}
\left|H(\mu_{t_0},\D_n)-\sum_{Q\in \D_n}a_Q\log \frac{1}{a_Q}\right| \le O_K(1).
\end{equation}
By \eqref{eq:pf:thm-pertubate-increase-2:1} and \eqref{eq:pf:thm-pertubate-increase-2:3.1},  we have
\begin{equation}\label{eq:pf:thm-pertubate-increase-2:4}
\frac{1}{n}\sum_{Q\in \D_n}a_Q\log \frac{1}{a_Q} =\dim \mu_{t_0} \pm \epsilon\pm \frac{O_K(1)}{n}.
\end{equation}
Observe that by Lemma \ref{lemma:almost-continuity-entropy:1} (3),   we can write
\[
\sum_{I\in \Lambda_n}p_I\log \frac{1}{p_I} =\sum_{Q\in \D_n}a_Q\log \frac{1}{a_Q} +\sum_{Q\in \D_n}a_Q\sum_{I\in A_Q}\frac{p_I}{a_Q}\log \frac{a_Q}{p_I}.
\]
In view of \eqref{eq:pf:thm-pertubate-increase-2:2} and \eqref{eq:pf:thm-pertubate-increase-2:4}, it follows  that 
\[
\frac{1}{n}\sum_{Q\in \D_n}a_Q\sum_{I\in A_Q}\frac{p_I}{a_Q}\log \frac{a_Q}{p_I}\ge \sd \mu_{t_0}-\dim\mu_{t_0}-o_{\epsilon,n}(1).
\]
Recall that we have assumed $\dim\mu_{t_0}<\min(1,\sd \mu_{t_0})$.

Note that for $Q\in \D_n$,  we have $\sum_{I\in A_Q}\frac{p_I}{a_Q}\log \frac{a_Q}{p_I}\le \log |\Lambda_n|\le Cn$, where $C$ is some constant depending only on $\F_{t_0}$ and $\Lambda$.   Since $\sum_{a\in \D_n}a_Q=1$,  by Markov's inequality there exist $\delta_1>0$,  depending only on $ \F_{t_0},  \sd \mu_{t_0}-\dim\mu_{t_0}$ and $\Lambda$,  and  $B_1\subset \D_n$ such that $\sum_{Q\in B_1}a_Q\ge \delta_1$ and for each $Q\in B_1$,
\begin{equation}\label{eq:pf:thm-pertubate-increase-2:6}
\frac{1}{n}\sum_{I\in A_Q}\frac{p_I}{a_Q}\log \frac{a_Q}{p_I} \ge \delta_1.
\end{equation}

We now estimate $H(\mu_{t_0},\D_{3nm})$ and $H(\mu_{t},\D_{3nm})$.  We shall use the following formulas which express $H(\mu_{t_0},\D_{3nm})$ (resp. $H(\mu_{t},\D_{3nm})$) in terms of $H(\nu_Q^{t_0},  \D_{3nm})$ (resp.  $H(\nu_Q^{t},  \D_{3nm})$),   $Q\in \D_n$.
\begin{lemma}\label{lemma:pf-thm-pertubate-increase-2:1}
We have
\begin{equation}\label{eq:pf:lemma:pf-thm-pertubate-increase-2:1:1}
H(\mu_{t_0},\D_{3nm}) = \sum_{Q\in \D_n}a_Q\log \frac{1}{a_Q}  +  \sum_{Q\in \D_n}a_Q H(\nu_Q^{t_0},  \D_{3nm}) +O_K(1).  
\end{equation}
and
\begin{equation}\label{eq:pf:lemma:pf-thm-pertubate-increase-2:1:2}
H(\mu_{t},\D_{3nm}) = \sum_{Q\in \D_n}a_Q\log \frac{1}{a_Q}  +  \sum_{Q\in \D_n}a_Q H(\nu_Q^{t},  \D_{3nm}) +O_K(1). 
\end{equation}
\end{lemma}
The proof of Lemma \ref{lemma:pf-thm-pertubate-increase-2:1} is postponed after the proof of Theorem \ref{thm:main:pertubate-increase:2}.  

By \eqref{eq:pf:thm-pertubate-increase-2:1}, we have $H(\mu_{t_0},\D_{3nm}) = 3nm (\dim \mu_{t_0}\pm \epsilon)$ and for each $Q\in \D_n$ and $I\in A_Q$,   $H(f_{I,t_0}\mu_{t_0},\D_{3nm}) = (3m-1)n (\dim \mu_{t_0}\pm \epsilon)$ (noting that $f_{I,t_0}$ has contraction ratio about $2^{-n}$). Thus by Lemma \ref{lemma:convexity-entropy:1} (1),  we have $H(\nu_Q^{t_0},  \D_{3nm}) \ge (3m-1)n (\dim \mu_{t_0} -  \epsilon)$.  In view of this,   \eqref{eq:pf:thm-pertubate-increase-2:3.1} and \eqref{eq:pf:lemma:pf-thm-pertubate-increase-2:1:1},  we infer that there exists $B_2\subset \D_n$ such that $\sum_{Q\in B_2}a_Q\ge 1-o_{\epsilon}(1)$ and for each $Q\in B_2$,  we have 
\begin{equation}\label{eq:pf:thm-pertubate-increase-2:7}
H(\nu_Q^{t_0},  \D_{3nm}) \le  (3m-1)n (\dim \mu_{t_0} +o_\epsilon(1)).
\end{equation}

We now need the following estimates about $H(\nu_Q^{t_0},  \D_{3nm}) $ and $H(\nu_Q^{t},  \D_{3nm}) $.  
\begin{lemma}\label{lemma:pf-thm-pertubate-increase-2:2}
For $Q\in \D_n$,  we have
\begin{equation}\label{eq:pf:lemma:pf-thm-pertubate-increase-2:2:1}
H(\nu_Q^{t},  \D_{3nm}) \ge  (3m-1)n (\dim \mu_{t_0}-o_\epsilon(1)).
\end{equation}
\end{lemma}

\begin{lemma}\label{lemma:pf-thm-pertubate-increase-2:3}
There exists $\delta_2>0$,  depending only on $\delta_1$,  $\beta$ and $\mu_{t_0}$, such that for each $Q\in B_1\cap B_2$,  we have
\begin{equation}\label{eq:pf:lemma:pf-thm-pertubate-increase-2:2:1}
H(\nu_Q^{t},  \D_{3nm}) \ge  (3m-1)n (\dim \mu_{t_0}+\delta_2).
\end{equation}
\end{lemma}
The proof of Lemmas \ref{lemma:pf-thm-pertubate-increase-2:2} and \ref{lemma:pf-thm-pertubate-increase-2:3}  are postponed after the proof of Theorem \ref{thm:main:pertubate-increase:2}.  

By Lemmas \ref{lemma:pf-thm-pertubate-increase-2:1}, \ref{lemma:pf-thm-pertubate-increase-2:2} and \ref{lemma:pf-thm-pertubate-increase-2:3},   we get 
\begin{equation}\label{eq:pf:thm-pertubate-increase-2:7.1}
\frac{1}{3nm} H(\mu_{t},\D_{3nm}) \ge \dim \mu_{t_0}+\delta_3
\end{equation}
for some $\delta_3>0$ depending only on $\delta_1, \beta$ and $\mu_{t_0}$.   
It is known  (see \cite[Section 3]{Peres-Solomyak}) that for any $M\ge 1$ there exists constant $C$ such that for any self-similar measure $\eta$ on $[-M,M]$ whose defining IFS $\Phi=\{\phi_i(x)=r_ix +s_i\}_i$ satisfying $\min_i|r_i|\ge 1/M$, the following holds for all $k\ge 1$:
\begin{equation}\label{eq:pf:thm-pertubate-increase-2:7.1.1}
\dim \eta \ge \frac{1}{k} H(\eta,\D_k)-\frac{C}{k}.
\end{equation}
We may assume that we initially work in a small neighbourhood $U$ of the parameter $t_0$ so that there exists absolute constant $M\ge 1$ such that for all $t\in U$ the attractor of $\mathcal{F}_t$ is included $[-M,M]$  and $\min_{i\in \Lambda}|r_{i,t}|\ge 1/M$. Thus for each  $t\in U$ the measure $\mu_t$ satisfies  \eqref{eq:pf:thm-pertubate-increase-2:7.1.1}. The absolute constant $C$ is independent from $\epsilon_0$ and $n$.
Using this and \eqref{eq:pf:thm-pertubate-increase-2:7.1}, we get
\[
\dim \mu_t\ge \dim \mu_{t_0} +\delta_3/2,
\]
provided $\epsilon_0$ was assumed small enough (thus $n$ is large enough). 
\end{proof}

Let us now give the proofs of Lemmas \ref{lemma:pf-thm-pertubate-increase-2:1}, \ref{lemma:pf-thm-pertubate-increase-2:2} and \ref{lemma:pf-thm-pertubate-increase-2:3}.   
\begin{proof}[Proof of Lemma \ref{lemma:pf-thm-pertubate-increase-2:1}]
Recall that $p=(p_i)_{i\in \Lambda}$ is a probability vector.   Denote by $\tilde{\mu}$ the Bernoulli measure $p^\N$ on $\Lambda^\N$.  
For $Q\in \D_n(\R)$, let 
\[
\mathcal{A}_Q=\left\{  \left(\cup_{I\in A_Q}[I]\right)\cap \pi_{t_0}^{-1} (Q'):   Q'\in \D_{3nm}(\R)  \right\}
\]
and $\mathcal{A} = \cup_{Q\in \D_n} \mathcal{A}_Q$. Let 
\[
\mathcal{C} = \left\{  \pi_{t_0}^{-1} (Q'):   Q'\in \D_{3nm}(\R) \right\}.
\]
Clearly, each element of $\mathcal{A}$ intersects only one element  of $\mathcal{C}$. On the other hand, since for each $I\in \Lambda_n$,   ${\rm diam}(f_{I,t_0}(K))\le {\rm diam}(K)\cdot 2^{-n}\cdot (\min_{i}|r_{i,t_0}|)^{-1}$,  by definition of $A_Q$, it follows that for each $x\in \R$, we have 
\[
\left|\left\{Q\in \D_n(\R): x\in \pi_{t_0}\left(\cup_{I\in A_Q}[I]\right)\right\}\right| \le\  \frac{{\rm diam}(K)}{ \min_{i}|r_{i,t_0}|}+1.
\]
This implies that each element of $\mathcal{C}$  intersects  at most $\left({\rm diam}(K)\cdot (\min_{i}|r_{i,t_0}|)^{-1} +1\right)$ elements  of $\mathcal{A}$. 
Thus, by Lemma \ref{lemma:almost-continuity-entropy:1} (1),  we have 
\begin{equation}\label{eq:pf:lemma-pf-thm-pertubate-increase-2:1:1}
\left| \sum_{e\in \mathcal{A}} \tilde{\mu}(e)\log  \tilde{\mu}(e)^{-1} - \sum_{e\in \mathcal{C}} \tilde{\mu}(e)\log  \tilde{\mu}(e)^{-1} \right| \le O_K(1). 
\end{equation}
Recall that $a_Q=\sum_{I\in A_Q}p_I=\sum_{e\in \mathcal{A}_Q}\tilde{\mu}(e)$.
By Lemma \ref{lemma:almost-continuity-entropy:1} (3),   we have 
\begin{eqnarray}
 \sum_{e\in \mathcal{A}} \tilde{\mu}(e)\log  \tilde{\mu}(e)^{-1}  & = & \sum_{Q\in \D_n} \tilde{\mu}\left(\bigcup_{e\in \mathcal{A}_Q}e\right)\log \tilde{\mu}\left(\bigcup_{e\in \mathcal{A}_Q}e\right)^{-1} +   \sum_{Q\in \D_n} a_Q\sum_{e\in \mathcal{A}_Q}\frac{\tilde{\mu}(e)}{a_Q} \log \frac{a_Q}{\tilde{\mu}(e)}\nonumber\\
& = & \sum_{Q\in \D_n}a_Q\log\frac{1}{a_Q} +\sum_{Q\in \D_n} a_Q H(\nu_Q^{t_0},\D_{3nm}). \label{eq:pf:lemma-pf-thm-pertubate-increase-2:1:2}
\end{eqnarray}
We also have 
\begin{equation}\label{eq:pf:lemma-pf-thm-pertubate-increase-2:1:3}
\sum_{e\in \mathcal{C}} \tilde{\mu}(e)\log  \tilde{\mu}(e)^{-1} =H(\mu_{t_0},\D_{3nm}).
\end{equation}
Combining \eqref{eq:pf:lemma-pf-thm-pertubate-increase-2:1:1}, \eqref{eq:pf:lemma-pf-thm-pertubate-increase-2:1:2} and \eqref{eq:pf:lemma-pf-thm-pertubate-increase-2:1:3},   we get the desired formula for $H(\mu_{t_0},\D_{3nm})$.  

 Now,  we turn to the formula for $H(\mu_{t},\D_{3nm})$.   For $Q\in \D_n(\R)$,  let 
\[
\mathcal{A}_Q^t=\left\{  \left(\cup_{I\in A_Q}[I]\right)\cap \pi_{t}^{-1} (Q'):   Q'\in \D_{3nm}(\R)  \right\}
\]
and $\mathcal{A} ^t= \cup_{Q\in \D_n} \mathcal{A}_Q^t$. Let 
\[
\mathcal{C}^t = \left\{  \pi_{t}^{-1} (Q'):   Q'\in \D_{3nm}(\R) \right\}.
\]
Recall that the IFSs $\{\F_t\}_t$ satisfy the condition \eqref{eq:definition-IFS-regularity:1}. Thus there exists $C>0$ such that for any $x\in \Lambda^\N$,  we have 
\[
|f_{x,t}(0)-f_{x,t_0}(0) |\le C|t-t_0|.
\] 
Since $|t-t_0| \le  2^{-(n-1)m}$ and $m=\lfloor \beta+3\rfloor$,  by replacing $m$ with a larger constant if necessary,  we may assume that we have $C|t-t_0|\le 2^{-n-2}$.  Using this, it is not hard to check that for each $x\in \R$, we have 
\[
\left|\left\{Q\in \D_n(\R): x\in \pi_{t}\left(\cup_{I\in A_Q}[I]\right)\right\}\right| \le {\rm diam}(K)\cdot (\min_{i}|r_{i,t_0}|)^{-1} +3.
\]
This implies that each element of $\mathcal{C}^t$  intersects  at most $({\rm diam}(K)\cdot (\min_{i}|r_{i,t_0}|)^{-1} +3)$ elements  of $\mathcal{A}^t$.  Clearly, each element of $\mathcal{A}^t$ intersects only one element  of $\mathcal{C}^t$.  The rest of the arguments is the same as in the case for $H(\mu_{t_0},\D_{3nm})$. 
\end{proof}  

\begin{proof}[Proof of Lemma \ref{lemma:pf-thm-pertubate-increase-2:2}]
{\color{black}In the proof of Lemma \ref{lemma:pf-thm-pertubate-increase-2:1} we have seen that for any $x\in \Lambda^\N$,  we have 
\[
|f_{x,t}(0)-f_{x,t_0}(0) |\le 2^{-n-2}.
\] } 
From this, we infer that the L\'evy-Prokhorov distance between $\mu_t$ and $\mu_{t_0}$ is 
\begin{equation}\label{eq:pf:lemma-pf-thm-pertubate-increase-2:2:1}
d(\mu_t,\mu_{t_0}) \le 2^{-n}.
\end{equation}
Recall that the measure $\mu_{t_0}$ has uniform entropy dimension $\dim \mu_{t_0}$ (recall \eqref{eq:pf:thm-pertubate-increase-2:2.1}),  in particular  we have 
\begin{equation}\label{eq:pf:lemma-pf-thm-pertubate-increase-2:2:2}
\mu_{t_0} \left( x:\left| \left\{ 1\le k\le n: \left|H(\mu_{t_0}^{\D_k(x)},\D_l)-l\dim\mu_{t_0}\right| \le l \epsilon \right\} \right| \ge n(1-\epsilon) \right) \ge 1-\epsilon.
\end{equation}
From \eqref{eq:pf:lemma-pf-thm-pertubate-increase-2:2:1} and \eqref{eq:pf:lemma-pf-thm-pertubate-increase-2:2:2},  we infer that for any $h(x)=rx+s$ with $|r| = 2^{-N\pm 1}$, we have 
\begin{equation}\label{eq:pf:lemma-pf-thm-pertubate-increase-2:2:3}
(h\mu_{t}) \left( x:\left| \left\{ N\le k\le N+n: \left|H((h\mu_{t})^{\D_k(x)},\D_l)-l\dim\mu_{t_0}\right| \le 2 l \epsilon \right\} \right| \ge n(1-2\epsilon) \right) \ge 1-2\epsilon.
\end{equation}
 
 We now estimate $H(\nu_Q^t,\D_{3nm})$.  Recall that $\nu_Q^t=\frac{1}{a_Q}\sum_{I\in A_Q}p_If_{I,t}\mu_t$.  By Lemma \ref{lemma:convexity-entropy:1} (1),  in order to prove $H(\nu_Q^t,\D_{3nm})\ge (3m-1)n(\dim\mu_{t_0}-o_\epsilon(1))$, we only need to show that for each $I\in A_Q $, 
 \[
H(f_{I,t}\mu_{t},\D_{3nm})\ge (3m-1)n(\dim\mu_{t_0}-o_\epsilon(1)).
 \]
Since the contraction ratio of $f_{I,t}$ is $2^{-n+o(n)}$ \footnote{Since the maps $t\mapsto r_{i,t}, i\in \Lambda$ are $C^2$,  there exists absolute  constant $C>0$ such that $|r_{i,t}-r_{i,t_0}|\le C|t-t_0|$.  Recall that  $|t-t_0|=2^{-nm\pm m}$.   By replacing $m$ with a larger constant if necessary,  we may assume that $|t-t_0|\le 2^{-nm/2}|r_{i,t_0}|$.  Hence $|r_{I,t}|=|r_{I,t_0}|(1\pm 2^{-nm/2})^{O(n)}=2^{-n+o(n)}$.},  we have 
\begin{equation}\label{eq:pf:lemma-pf-thm-pertubate-increase-2:2:4}
H(f_{I,t}\mu_{t},\D_{3nm}) = H(\mu_{t},\D_{(3m-1)n}) +o(n). 
\end{equation}
We recall that for any measure $\eta\in \cP(\R)$,  the following formula holds (\cite[Lemma 3.4]{Hochman2014})
\[
H(\eta,\D_n)=\sum_{k=0}^{n-1}\sum_{Q\in \D_k}\eta(Q)\frac{H(\eta^Q,\D_l)}{l} +O(l),
\]
where the constant involving in $O(l)$ depends on the diameter of the support of $\eta$.  Using this we can write 
\begin{equation}\label{eq:pf:lemma-pf-thm-pertubate-increase-2:2:5}
H(\mu_t,\D_{(3m-1)n})=\sum_{j=0}^{3m-2}\ \ \sum_{k=nj}^{n(j+1)-1}\sum_{Q\in \D_k}\mu_t(Q)\frac{H(\mu_t^Q,\D_l)}{l} +O_K(l).
\end{equation}
For $j=0,\ldots,  3m-2$,  let 
\[\Lambda^t_{nj}=\{I\in \Lambda^*: |r_{I,t}|\le 2^{-jn}<|r_{I^-,t}|\}.\]
Then we have 
\begin{equation}\label{eq:pf:lemma-pf-thm-pertubate-increase-2:2:6}
\mu_t=\sum_{I\in \Lambda^t_{jn}} p_If_{I,t}\mu_t. 
\end{equation}
For $k\in \{nj,\ldots, n(j+1)-1\}$ and $Q\in \D_k$, using \eqref{eq:pf:lemma-pf-thm-pertubate-increase-2:2:6},  we have 
\begin{eqnarray}
  \mu_t(Q)H(\mu_t^Q,\D_l) & = & \sum_{Q'\in \D_{k+l}}\mu_t(Q') \log \frac{\mu_t(Q)}{\mu_t(Q')} \nonumber\\
   & = &\sum_{I\in \Lambda_{jn}}p_I\sum_{Q'\in \D_{k+l}}(f_{I,t}\mu_t)(Q')  \log \frac{\mu_t(Q)}{\mu_t(Q')} \nonumber\\
   & = & \sum_{I\in \Lambda_{jn}}p_I(f_{I,t}\mu_t)(Q)\sum_{Q'\in \D_{k+l}}\frac{(f_{I,t}\mu_t)(Q')}{(f_{I,t}\mu_t)(Q)}  \log \frac{\mu_t(Q)}{\mu_t(Q')}   \nonumber\\
& \ge &  \sum_{I\in \Lambda_{jn}}p_I(f_{I,t}\mu_t)(Q) H\left((f_{I,t}\mu_t)^{Q},\D_l\right),  \label{eq:pf:lemma-pf-thm-pertubate-increase-2:2:7}
\end{eqnarray}
where for the last inequality we have used the Gibbs inequality: $\sum_i a_i\log \frac{1}{b_i}\ge \sum_i a_i\log \frac{1}{a_i}$ for any probability vectors $(a_i)_i$ and $(b_i)_i$ with $b_i=0\Rightarrow a_i=0$.  
Now,  from  \eqref{eq:pf:lemma-pf-thm-pertubate-increase-2:2:3},  \eqref{eq:pf:lemma-pf-thm-pertubate-increase-2:2:5} and \eqref{eq:pf:lemma-pf-thm-pertubate-increase-2:2:7},  we infer that (writing $n'=(3m-1)n$)
\begin{equation}\label{eq:pf:lemma-pf-thm-pertubate-increase-2:2:8}
{\Small
\mu_{t} \left( x:\frac{1}{n'}\left| \left\{ 1\le k \le  n': H(\mu_{t}^{\D_k(x)},\D_l)\ge l(\dim\mu_{t_0}  -o_ \epsilon(1)) \right\} \right| \ge 1-o_{\epsilon}(1) \right) \ge 1-o_{\epsilon}(1).
}
\end{equation}
It follows that we have
\[
H(\mu_t,\D_{(3m-1)n}) \ge (3m-1)n(\dim\mu_{t_0}  -o_ \epsilon(1)).
\]
This completes the proof of Lemma \ref{lemma:pf-thm-pertubate-increase-2:2}.
\end{proof}

\begin{proof}[Proof of Lemma \ref{lemma:pf-thm-pertubate-increase-2:3}]

Let us fix $Q\in B_1\cap B_2$.  Denote $R=\{r_{I,t_0}:I\in A_Q\}$.   The   cardinality of $R$ is bounded by $n^{C}$ for some constant $C$ only depending on $\Lambda$ and $\F_{t_0}$.
For $r\in R$, let $A_{Q,r}=\{I\in A_Q: r_{I,t_0}=r\}$.   Let 
\[
a_{Q,r} = \sum_{I\in A_{Q,r}}p_I,  \  \   \nu_{Q,r} ^{t_0}= \frac{1}{a_{Q,r}}\sum_{I\in A_{Q,r}} p_I f_{I,t_0} \mu_{t_0}. 
\]
Since $\nu_{Q}^{t_0}=\sum_{r\in R}\frac{a_{Q,r}}{a_Q}\nu_{Q,r}^{t_0}$,   by Lemma \ref{lemma:convexity-entropy:1} (1),  we have 
\[
H(\nu_{Q}^{t_0},\D_{3nm}) \ge \sum_{r\in R} \frac{a_{Q,r}}{a_Q}H(\nu_{Q,r} ^{t_0},\D_{3nm}). 
\]
Since $Q\in B_2$,  we have (recall \eqref{eq:pf:thm-pertubate-increase-2:7})
\begin{equation}\label{eq:pf:lemma-pf-thm-pertubate-increase-2:3:0}
H(\nu_Q^{t_0},  \D_{3nm}) \le  (3m-1)n (\dim \mu_{t_0} +o_\epsilon(1)).
\end{equation}
 For each $r\in R$, the measure $\nu_{Q,r} ^{t_0}$ is the convolution of the measure $r\mu_{t_0}$ with another probability measure,  and $|r|$ is about $2^{-n}$,  hence we have
 \begin{equation}\label{eq:pf:lemma-pf-thm-pertubate-increase-2:3:0.1}
H(\nu_{Q,r}^{t_0},  \D_{3nm}) \ge H(r\mu_{t_0},\D_{3nm})-O(1) \ge (3m-1)n (\dim \mu_{t_0} -o_\epsilon(1)).
\end{equation}
By \eqref{eq:pf:lemma-pf-thm-pertubate-increase-2:3:0} and \eqref{eq:pf:lemma-pf-thm-pertubate-increase-2:3:0.1},  we infer that there exists $R'\subset R$ such that $\sum_{r\in R'} \frac{a_{Q,r}}{a_Q} \ge 1-o_\epsilon(1) $ and for each $r\in R'$ we have
\begin{equation}\label{eq:pf:lemma-pf-thm-pertubate-increase-2:3:1}
H(\nu_{Q,r}^{t_0},  \D_{3nm}) \le  (3m-1)n (\dim \mu_{t_0} +o_\epsilon(1)).
\end{equation}
Note that for each $r\in R$, we have
\begin{equation}\label{eq:pf:lemma-pf-thm-pertubate-increase-2:3:2}
\nu_{Q,r}^{t_0} =\eta_{Q,r}^{t_0}* r\mu_{t_0}
\end{equation}
where $\eta_{Q,r}^{t_0}=\frac{1}{a_{Q,r}}\sum_{I\in A_{Q,r}}p_I \delta_{f_{I,t_0}(0)}$.
Recall that the measure $\mu_{t_0} $ has uniform entropy dimension $\dim\mu_{t_0}$ (recall \eqref{eq:pf:thm-pertubate-increase-2:2.1}).  Hence in view of \eqref{eq:pf:lemma-pf-thm-pertubate-increase-2:3:1} and \eqref{eq:pf:lemma-pf-thm-pertubate-increase-2:3:2}, by Hochman's inverse theorem  (Theorem \ref{theorem:Hochman-inverse-thm}),  we must have 
\begin{equation}\label{eq:pf:lemma-pf-thm-pertubate-increase-2:3:3}
H(\eta_{Q,r}^{t_0},\D_{3nm}) = o_\epsilon(n).
\end{equation}
Since the cardinality of $R$ is bounded by  $n^C=2^{o(n)}$ and $\sum_{r\in R'}\frac{a_{Q,r}}{a_Q}\ge 1-o(1)$,  we get 
\begin{equation}\label{eq:pf:lemma-pf-thm-pertubate-increase-2:3:4}
H\left(\frac{1}{a_{Q}}\sum_{I\in A_{Q}}p_I \delta_{f_{I,t_0}(0)},\D_{3nm}\right) = H\left(\sum_{r\in R}\frac{a_{Q,r}}{a_Q}\eta_{Q,r}^{t_0},\D_{3nm}\right) = o(n).
\end{equation}

Now we need the following lemma whose proof is postponed after the proof of Lemma \ref{lemma:pf-thm-pertubate-increase-2:3}. 
\begin{lemma}\label{lemma:pf-thm-pertubate-increase-2:4}
Let $\delta>0$. There exists $\rho=\rho(\delta,\beta)>0$ such that the following holds for all $n$ large enough: Let $A\subset \Lambda_n$ and $(q_I)_{I\in A}$ be a probability vector.  Suppose that 
\[
\frac{1}{n}\sum_{I\in A} q_I\log \frac{1}{q_I} \ge \delta.
\]
Let $t_1,t_2\in J$ be such that $|t_1-t_2|=2^{-nm\pm m}$.  Then we have either 
\[
\frac{1}{n}H\left(\sum_{I\in A}q_I\delta_{f_{I,t_1}(0)}, \D_{3nm}\right) \ge \rho
\]
or 
\[
\frac{1}{n}H\left(\sum_{I\in A}q_I\delta_{f_{I,t_2}(0)}, \D_{3nm}\right) \ge \rho.
\]
\end{lemma}

Let us continue the proof of Lemma \ref{lemma:pf-thm-pertubate-increase-2:3}.  Recall that since $Q\in B_1$, by \eqref{eq:pf:thm-pertubate-increase-2:6},  we have
\[
\frac{1}{n}\sum_{I\in A_Q}\frac{p_I}{a_Q}\log \frac{a_Q}{p_I} \ge \delta_1.
\]
By \eqref{eq:pf:lemma-pf-thm-pertubate-increase-2:3:4} and the fact $|t_0-t|=2^{-nm\pm m}$,  it follows from Lemma \ref{lemma:pf-thm-pertubate-increase-2:4} that there exists $\rho_1>0$ depending only on $\delta_1$ and $\beta$ such that 
\begin{equation}\label{eq:pf:lemma-pf-thm-pertubate-increase-2:3:5}
\frac{1}{n}H\left(\frac{1}{a_{Q}}\sum_{I\in A_{Q}}p_I \delta_{f_{I,t}(0)},\D_{3nm}\right)  \ge \rho_1.
\end{equation}

Denote $R_t=\{r_{I,t} : I\in A_Q\}$.  Again,  the cardinality of $R_t$ is bounded by $2^{o(n)}$.   For each $r\in R_t$,  let $A_{Q,r}^t=\{I\in A_Q: r_{I,t} = r\}$.  Let 
\[
a_{Q,r}^t=\sum_{I\in A_{Q,r}^t}p_I \  \textrm{ and } \  \nu_{Q,r}^t=\frac{1}{a_{Q,r}^t} \sum_{I\in A_{Q,r}^t} p_If_{I,t}\mu_t.
\]
By Lemma \ref{lemma:convexity-entropy:1} (2), we have 
\[
H\left(\frac{1}{a_Q}\sum_{I\in A_Q}p_I \delta_{f_{I,t}(0)},  \D_{3nm}\right) \le \sum_{r\in R_t} \frac{a_{Q,r}^t}{a_Q}\log \frac{a_Q}{a_{Q,r}^t} +\sum_{r\in R_t}a_{Q,r}^t H\left(\frac{1}{a_{Q,r}^t}\sum_{I\in A_{Q,r}^t}p_I \delta_{f_{I,t}(0)},  \D_{3nm}\right)
\]
Since $\sum_{r\in R_t} \frac{a_{Q,r}^t}{a_Q}\log \frac{a_Q}{a_{Q,r}^t}\le \log |R_t|=o(n)$,  in view of \eqref{eq:pf:lemma-pf-thm-pertubate-increase-2:3:5}, we get
\[
\frac{1}{n}\sum_{r\in R_t}a_{Q,r}^t H\left(\frac{1}{a_{Q,r}^t}\sum_{I\in A_{Q,r}^t}p_I \delta_{f_{I,t}(0)},  \D_{3nm}\right) \ge \rho_1-o(1).
\]
From this,  we infer that there exist $\rho_2>0$, depending on $\rho_1$, and $R_t'\subset R_t$ such that $\sum_{r\in R_t' } a_{Q,r}^t\ge \rho_2$ and for each $r\in R_t'$,  we have
\begin{equation}\label{eq:pf:lemma-pf-thm-pertubate-increase-2:3:6}
\frac{1}{n}H\left(\frac{1}{a_{Q,r}^t}\sum_{I\in A_{Q,r}^t}p_I \delta_{f_{I,t}(0)},  \D_{3nm}\right) \ge \rho_2.
\end{equation}
Now we shall show that for each $r\in R_t'$,  we have
\begin{equation}\label{eq:pf:lemma-pf-thm-pertubate-increase-2:3:6.0}
H\left(\frac{1}{a_{Q,r}^t}\sum_{I\in A_{Q,r}^t}p_I f_{I,t}\mu_t,  \D_{3nm}\right) \ge  (3m-1)n(\dim\mu_t+\rho_2')
\end{equation}
where  $\rho_2'>0$ is an absolute constant only depending on $\rho_2$ and $l$.
Note that we have
\begin{equation}\label{eq:pf:lemma-pf-thm-pertubate-increase-2:3:6.1}
\frac{1}{a_{Q,r}^t}\sum_{I\in A_{Q,r}^t}p_I f_{I,t}\mu_t = \left(\frac{1}{a_{Q,r}^t}\sum_{I\in A_{Q,r}^t}p_I \delta_{f_{I,t}(0)}\right) * r\mu_t.
\end{equation}
Let us suppose, towards a contradiction,  that for some very small $\tau\ll 1$ we have
\begin{equation}\label{eq:pf:lemma-pf-thm-pertubate-increase-2:3:6.2}
H\left(\frac{1}{a_{Q,r}^t}\sum_{I\in A_{Q,r}^t}p_I f_{I,t}\mu_t,  \D_{3nm}\right) \le  (3m-1)n(\dim\mu_t+\tau).
\end{equation}
In the proof of Lemma \ref{lemma:pf-thm-pertubate-increase-2:2}, we have seen that 
$H(r\mu_t,\D_{3nm})\ge (3m-1)n(\dim\mu_{t_0}-o(1))$.  In view of this, together with \eqref{eq:pf:lemma-pf-thm-pertubate-increase-2:3:6},  \eqref{eq:pf:lemma-pf-thm-pertubate-increase-2:3:6.1} and \eqref{eq:pf:lemma-pf-thm-pertubate-increase-2:3:6.2}, we infer from Hochman's inverse theorem (Theorem \ref{theorem:Hochman-inverse-thm}) that there exists $\rho_3$, depending  only on $\rho_2$,  and $D\subset \{1,2, \ldots,3nm\}$ such that $|D|\ge 3nm\rho_3$ and for each $k\in D$, we have
\begin{equation}\label{eq:pf:lemma-pf-thm-pertubate-increase-2:3:7}
(r\mu_t)\left(x: \frac{1}{l}H\left((r\mu_t)^{\D_k(x)}, \D_l\right) \ge 1-o(1) \right) \ge 1-o(1).
\end{equation}
On the other hand,  recall that in the proof of Lemma \ref{lemma:pf-thm-pertubate-increase-2:2}, we have proved \eqref{eq:pf:lemma-pf-thm-pertubate-increase-2:2:8} from which it follows that  (noting that $|r|$ is about $2^{-n}$)
\begin{equation}\label{eq:pf:lemma-pf-thm-pertubate-increase-2:3:8}
\sum_{k=1}^{(3m-1)n}(r\mu_t)\left(x: \frac{1}{l}H\left((r\mu_t)^{\D_k(x)}, \D_l\right) \ge \dim\mu_{t_0}-o(1) \right) \ge (3m-1)n(1-o(1)).
\end{equation}
Combining \eqref{eq:pf:lemma-pf-thm-pertubate-increase-2:3:7} and \eqref{eq:pf:lemma-pf-thm-pertubate-increase-2:3:8},  we obtain
\[
H(r\mu_{t},\D_{3nm}) \ge (3m-1)n (\dim\mu_{t_0}+\rho_4)
\]
for some $\rho_4>0$ depending only on $\rho$ and $l$. This, together with \eqref{eq:pf:lemma-pf-thm-pertubate-increase-2:3:6.1},  implies that 
\[
H\left(\frac{1}{a_{Q,r}^t}\sum_{I\in A_{Q,r}^t}p_I f_{I,t}\mu_t ,\D_{3nm}\right) \ge H(r\mu_t,\D_{3nm})-O(1) \ge (3m-1)n(\dim\mu_{t_0}+\rho_4-o(n)).
\]
This is a contradiction to \eqref{eq:pf:lemma-pf-thm-pertubate-increase-2:3:6.2} if $\tau$ is too small in terms of $\rho_2$ and $l$.   Thus we must have \eqref{eq:pf:lemma-pf-thm-pertubate-increase-2:3:6.0}.

Now since $\nu_{Q}^t=\sum_{r\in R_t}\frac{a_{Q,r}^t}{a_Q}\nu_{Q,r}^t$ and $\sum_{r\in R_t}a_{Q,r}^t\ge \rho_2$,  therefore by \eqref{eq:pf:lemma-pf-thm-pertubate-increase-2:3:6.0} and Lemma \ref{lemma:convexity-entropy:1} (1),  we get
\[
H(\nu^t_Q,\D_{3nm}) \ge (3m-1)n(\dim\mu_t+\delta_2)
\]
for some $\delta_2$ depending on $\delta_1, \beta$ and $\mu_{t_0}$.  This finishes the proof of Lemma \ref{lemma:pf-thm-pertubate-increase-2:3}.
\end{proof}

\begin{proof}[Proof of Lemma \ref{lemma:pf-thm-pertubate-increase-2:4}]
Let us fix a small $\rho<\delta/2$, that we shall specify later.  We suppose 
\begin{equation}\label{eq:pf:lemma-pf-thm-pertubate-increase-2:4:2}
\frac{1}{n}H\left( \sum_{I\in A} q_I\delta_{f_{I,t_1}(0)} ,\D_{3nm}\right) <\rho.
\end{equation}
For each $Q\in \D_{3nm}$, let 
\[
A_Q=\{I\in A:f_{I,t_1}(0)\in Q\} \ \textrm{ and }  q_Q=\sum_{I\in A_Q}q_I.
\]
Then by Lemma \ref{lemma:almost-continuity-entropy:1} (3) we can write 
\begin{equation}\label{eq:pf:lemma-pf-thm-pertubate-increase-2:4:3}
\sum_{I\in A} q_I\log \frac{1}{q_I} =\sum_{Q\in\D_{3nm}} q_Q\log \frac{1}{q_Q} +\sum_{Q\in \D_{3nm}}q_Q\left( \sum_{I\in A_Q}\frac{q_I}{q_Q}\log \frac{q_Q}{q_I} \right). 
\end{equation}
Note that we have
\begin{equation}\label{eq:pf:lemma-pf-thm-pertubate-increase-2:4:4}
H\left( \sum_{I\in A} q_I\delta_{f_{I,t_1}(0)} ,\D_{3nm} \right) = \sum_{Q\in\D_{3nm}} q_Q\log \frac{1}{q_Q} .
\end{equation}
Recall that by our hypothesis,  $\sum_{I\in A} q_I\log \frac{1}{q_I}\ge \delta$.  Since $\rho<\delta/2$,  by \eqref{eq:pf:lemma-pf-thm-pertubate-increase-2:4:2}, \eqref{eq:pf:lemma-pf-thm-pertubate-increase-2:4:3} and \eqref{eq:pf:lemma-pf-thm-pertubate-increase-2:4:4},  we get
\[
\frac{1}{n} \sum_{Q\in \D_{3nm}}q_Q\left( \sum_{I\in A_Q}\frac{q_I}{q_Q}\log \frac{q_Q}{q_I} \right) \ge \delta/2.
\]
By Markov's inequality, we deduce that there exist $M=M(\delta,|\Lambda|)>0$ and $B\subset \D_{3nm}$ such that $\sum_{Q\in B}q_Q>1/M$ and for each $Q\in B$,  we have
\begin{equation}\label{eq:pf:lemma-pf-thm-pertubate-increase-2:4:5}
\frac{1}{n} \sum_{I\in A_Q}\frac{q_I}{q_Q}\log \frac{q_Q}{q_I} \ge 1/M.
\end{equation}

We now show that for any $Q\in \D_{3nm}$ and any $I,J\in A_Q$ with $I\neq J$,  we have
\begin{equation}\label{eq:pf:lemma-pf-thm-pertubate-increase-2:4:5.1}
|f_{I,t_2}(0)-f_{J,t_2}(0)|\ge 2^{-2nm}.
\end{equation}
This is consequence of the transversal property of $\{\F_{t}\}_t$.
Let us fix $Q\in \D_{3nm}$ and any $I,J\in A_Q$ with $I\neq J$.  By definition of $A_Q$,  we have
\[
|f_{I,t_1}(0)-f_{J,t_1}(0)|\le 2^{-3nm} < 2^{-\beta n}.
\]
Let us pick any $a\in \Lambda$ and let $x_I,x_J\in \Lambda^\N$ be defined as $x_I=Ia^\infty$ and $x_J=Ja^\infty$.  Then we have
\[
\Delta_{x_I,y_J}(t) = f_{x_I,t}(0)-f_{x_J,t}(0) = f_{I,t}(0)-f_{J,t}(0).
\]
Since $\{\F_t\}_t$ satisfies the $\beta$-transversality condition and $\left|\Delta_{x_I,y_J}(t_1)\right| < 2^{-n\beta}$,   we must have
\[
\left|\Delta_{x_I,y_J}'(t_1)\right| \ge 2^{-n\beta}.
\]
By the mean value theorem and the condition  $\left|\Delta_{x_I,y_J}''(t)\right| \le C$ (recall \eqref{eq:definition-IFS-regularity:1}),  we get for all $t$ with $|t-t_1|\le 2^{-nm+ m}$,
\[
\left|\Delta_{x_I,y_J}'(t)\right|  \ge \left|  \left|\Delta_{x_I,y_J}'(t_1)\right| -2^{-nm+ m}C \right|\ge 2^{-n\beta-2}.
\]
Recall that $m=\lfloor\beta+3\rfloor\ge \beta+2$.
By the mean value theorem again,  and the fact $|\Delta_{x_I,y_J}(t_1)|\le 2^{-3nm}$,  we get
\[
|\Delta_{x_I,y_J}(t_2)| \ge \left| |\Delta_{x_I,y_J}'(\tilde{t})||t_2-t_1| -  |\Delta_{x_I,y_J}(t_1)|  \right| \ge 2^{-2nm},
\]
where $\tilde{t}$ is some number between $t_1$ and $t_2$. Thus we have proved \eqref{eq:pf:lemma-pf-thm-pertubate-increase-2:4:5.1}.

Now by \eqref{eq:pf:lemma-pf-thm-pertubate-increase-2:4:5} and \eqref{eq:pf:lemma-pf-thm-pertubate-increase-2:4:5.1},  we have  
\begin{equation}\label{eq:pf:lemma-pf-thm-pertubate-increase-2:4:6}
H\left(\frac{1}{q_Q}\sum_{I\in A_Q}q_I\delta_{f_{I,t_2}(0)},\D_{3nm}\right)=
\sum_{I\in A_Q}\frac{q_I}{q_Q}\log \frac{q_Q}{q_I} \ge \frac{n}{M}.
\end{equation}
Note that we have
\[
\sum_{I\in A}q_I\delta_{f_{I,t_2}(0)} = \sum_{Q\in \D_{3nm}} q_Q\left(\frac{1}{q_Q}\sum_{I\in A_Q}q_I\delta_{f_{I,t_2}(0)}\right).
\]
Since $\sum_{Q\in B}q_Q\ge 1/M$ and for each $Q\in B$ we have \eqref{eq:pf:lemma-pf-thm-pertubate-increase-2:4:6},  by Lemma \ref{lemma:convexity-entropy:1} (1),  we get
\[
\frac{1}{n}H\left(\sum_{I\in A}q_I\delta_{f_{I,t_2}(0)},\D_{3nm}\right) \ge \rho
\]
provided $\rho$ is assumed small enough in terms of $M$.

\end{proof}

\section{Projections of CP-distributions}\label{section: proj-CP-distributions:1}


In this section, we present the proof of Theorem \ref{thm:proj-thm-cp-distribution:-2}. The proof strategy is similar to that of Theorem \ref{thm:main:pertubate-increase:1}, but it is technically more complicated due to the lack of convolution structures when examining the local scenery of the projected measures.

We first recall some basic properties concerning the dimensions of CP-distribution measures and their projections in Section \ref{subsection: properties-of-CP-dist:1}. The main ingredient in the proof of Theorem \ref{thm:proj-thm-cp-distribution:-2} is an entropy-increasing result from additive combinatorics. This result follows from Hochman's inverse theorem and the Balog-Szemerédi-Gowers theorem, with details provided in Section \ref{subsection: entropy-increasing-theorem:1}.
The final proof of Theorem \ref{thm:proj-thm-cp-distribution:-2} is presented in Section \ref{subsection: proof-of-proj-theorem-cp-dist:1}, following the preparations in Sections \ref{subsection: properties-of-CP-dist:1} and \ref{subsection: entropy-increasing-theorem:1}.

\subsection{Properties  of CP-distributions}\label{subsection: properties-of-CP-dist:1}

The CP-distribution theory has its gem in pioneering work of Furstenberg \cite{Furstenberg69}, initially as a tool to investigate intersections of Cantor sets.  Recently, a more systematic study of CP-distributions was initiated by Furstenberg \cite{Furstenber2008}, with further developments by Gavish \cite{Gavish}, Hochman \cite{Hochman2010-fractal-distributions}, Hochman and Shmerkin \cite{HS2012} and others.  Let us first recall some basic concepts related to this theory.

Recall that $\cP(X)$ denotes the space of Borel probability measures on a space  $X$.  For $D\in \D_n(\R^2)$,  $S_{D}$ is the unique orientation-preserving homothety sending $D$ to $[0,1)^2$.
\begin{definition}
The magnification operator $M :\cP([0,1]^2)\times [0,1]^2 \to \cP([0,1]^2)\times [0,1]^2$ is defined as  follows
\[
M(\mu, x) = (\mu^{\D_1(x)}, S_{\D_1(x)}(x)).
\]
\end{definition}
The operator $M$ is defined on pairs $(\mu,x)$ with $\mu(\D_1(x))>0$.    
\begin{definition}
A distribution $Q$ on $\cP([0,1]^2)\times [0,1]^2$
 is called adapted if for every $f \in C(\cP([0,1]^2)\times [0,1]^2)$,
\[
\int f(\mu, x)d Q(\mu, x) = \int \left(\int f(\mu, x)d\mu(x)\right) d Q(\mu).
\]
\end{definition}
Note that in the above, we used $d Q(\mu)$ instead of $dQ(\mu, x)$  since $x$   is not involved.  We shall often implicitly identify $Q$ with its marginal.

In other words, $Q$ is adapted if, conditioned on the measure component being $\mu$, the
point component $x$ is distributed according to $\mu$. In particular, if a property holds for
$Q$-a.e. $(\mu, x)$ and $Q$ is adapted, then this property holds for $Q$-a.e. $\mu$ and $\mu$-a.e. $x$.

\begin{definition}
A distribution $Q$ on $\cP([0,1]^2)\times [0,1]^2$ is a CP-distribution if it is
$M$-invariant and adapted.
\end{definition}
A CP-distribution $Q$ is ergodic if the measure preserving system $(\cP([0,1]^2)\times [0,1]^2, Q,M)$ is ergodic in
the usual sense. If it is not ergodic, then we can consider its ergodic decomposition.

We shall use the following properties about CP-distributions.

\begin{proposition}\label{prop:dim-projection-cp-distribution:1}
Let $Q$ be an ergodic CP-distribution on $\R^2$ and $\pi\in G(2,1)$.  Let $\alpha=\dim Q$ and $\beta=\dim \pi Q$. Then for $Q$-a.e.  $\mu$,  the measures $\mu$ and $\pi \mu$ are exact dimensional with $\dim \mu=\alpha$ and $\dim \pi \mu=\beta$.   In particular, for any $\epsilon>0$,  there is  $A$ such that $Q(A)>1-\epsilon$,  and for any $ \mu\in A$,  there exist $D$ and $r_0$ with $\mu(D)>1-\epsilon$ and 
\begin{eqnarray}
 \mu(B(x,r)) & = & r^{\alpha\pm \epsilon} \  \textrm{ for } x\in D,  r\le r_0,  \label{eq:prop:dim-projection-cp-distribution:1:1}\\
\pi\mu(B(\pi(x),r)) & = & r^{\beta\pm \epsilon} \  \textrm{ for } x\in D,  r\le r_0.\label{eq:prop:dim-projection-cp-distribution:1:2}
\end{eqnarray}
Moreover, for any $\epsilon>0$, there exists $l_1\in \N$ such that for all $l\ge l_1$, there is $A$ such that $Q(A)>1-\epsilon$ and for every $\mu\in A$, we have
\begin{eqnarray}
\mu\left( x: \frac{1}{l}H\left(\mu^{\D_l(x)},\D_l\right)  =  \alpha\pm \epsilon \right) &>& 1-\epsilon, \label{eq:prop:dim-projection-cp-distribution:1:3}\\
\mu\left( x: \frac{1}{l}H\left(\pi\left(\mu^{\D_l(x)}\right),\D_l\right) =  \beta\pm \epsilon \right) &>& 1-\epsilon.\label{eq:prop:dim-projection-cp-distribution:1:4}
\end{eqnarray}
\end{proposition}

\begin{proof}
By \cite[Theorem 1.22]{Hochman2010-fractal-distributions}, we know that for $Q$-a.e. $\mu$,  both $\mu$ and $\pi\mu$ are exact dimensional with $\dim \mu=\dim Q=\alpha$ and $\dim \pi\mu=\dim \pi Q=\beta$.   By Egorov's theorem, we have \eqref{eq:prop:dim-projection-cp-distribution:1:1} and \eqref{eq:prop:dim-projection-cp-distribution:1:2}. 

We now proceed to the proof of \eqref{eq:prop:dim-projection-cp-distribution:1:3} and \eqref{eq:prop:dim-projection-cp-distribution:1:4}.   For $\epsilon>0$ and $l\ge 1$, let 
\begin{eqnarray*}
& & F_{l,\epsilon} = \left\{ \mu\in \cP([0,1]^2): \frac{1}{l}H(\mu,\D_l(\R^2)) = \alpha\pm \epsilon  \right\},\\
& & F_{l,\epsilon}^\pi = \left\{ \mu\in \cP([0,1]^2): \frac{1}{l}H(\pi\mu,\D_l(\R)) = \beta\pm \epsilon  \right\}.
\end{eqnarray*}
Since for $Q$-a.e. $\mu$, both $\mu$ and $\pi\mu$ are exact dimensional with $\dim \mu=\alpha$ and $\dim \pi\mu=\beta$,  it follows that for $l$ large enough we have  
\[
Q(F_{l,\epsilon})>1-o_{l}(1) \  \textrm{ and }  \  Q(F_{l,\epsilon}^\pi)>1-o_{l}(1).
\]
Recall that $M$ is the magnification operator defined on $\cP([0,1]^2)\times [0,1]^2$.  By adaptedness of $Q$, we have
\begin{eqnarray}
& & Q\left(M^{-l}\left(F_{l,\epsilon}\times [0,1]^2\right)\right) = \int \mu\left(x: \mu^{\D_l(x)}\in F_{l,\epsilon}\right) d Q(\mu),  \label{eq:pf:prop:dim-projection-cp-distribution:1:1}\\
& & Q\left(M^{-l}\left(F_{l,\epsilon}^\pi\times [0,1]^2\right)\right) = \int \mu\left(x: \mu^{\D_l(x)}\in F_{l,\epsilon}^\pi\right) d Q(\mu).\label{eq:pf:prop:dim-projection-cp-distribution:1:2}
\end{eqnarray}
Since $Q$ is $M$-invariant, we have 
\begin{eqnarray}
& & Q\left(M^{-l}\left(F_{l,\epsilon}\times [0,1]^2\right)\right) =Q\left(F_{l,\epsilon}\times [0,1]^2\right)=Q\left(F_{l,\epsilon}\right)>1-o_l(1), \label{eq:pf:prop:dim-projection-cp-distribution:1:3}\\
& & Q\left(M^{-l}\left(F_{l,\epsilon}^\pi\times [0,1]^2\right)\right) = Q\left(F_{l,\epsilon}^\pi\times [0,1]^2\right)=Q\left(F_{l,\epsilon}^\pi\right)>1-o_l(1).\label{eq:pf:prop:dim-projection-cp-distribution:1:4}
\end{eqnarray}
Combining \eqref{eq:pf:prop:dim-projection-cp-distribution:1:1} and \eqref{eq:pf:prop:dim-projection-cp-distribution:1:3}, and by Markov's inequality,  we infer that there exists $F\subset \cP([0,1]^2)$ such that $Q(F)>1-o_{l}(1)$ and for each $\mu\in F$,  we have   $\mu\left(x: \mu^{\D_l(x)}\in F_{l,\epsilon}\right) > 1-o_{l}(1)$.  In particular,  for each $\mu\in F$,
\[
\mu\left( x: \frac{1}{l}H\left(\mu^{\D_l(x)},\D_l\right)  =  \alpha\pm \epsilon \right) > 1-o_l(1). 
\]
Assuming $l$ is large enough so that $o_l(1)<\epsilon$, we obtain \eqref{eq:prop:dim-projection-cp-distribution:1:3}.  By a similar argument,   using 
\eqref{eq:pf:prop:dim-projection-cp-distribution:1:2} and \eqref{eq:pf:prop:dim-projection-cp-distribution:1:4},  we obtain \eqref{eq:prop:dim-projection-cp-distribution:1:4}. 

\end{proof}

\begin{definition}\label{definition:spreading measures}
A measure $\eta\in \cP(\R)$ is $(n,l,\epsilon)$-dyadic spreading if  there exists $t\in \R$ such that the measure $\tilde{\eta}=\eta+t$ satisfies
\[
\tilde{\eta}\left(x: \frac{1}{n}| \{1\le k\le n: \tilde{\eta}(\D_k(x))\le 2 \tilde{\eta}(\D_{k+l}(x))\} |<\epsilon\right) >1-\epsilon.
\]
\end{definition}

Recall that if   $\mu\in \cP(\R^2)$ and $\pi\in G(2,1)$, then for $\pi\mu$-a.e.  $x$,  $\mu_{\pi^{-1}(x)}$ denotes the conditional measure of $\mu$ on the fiber $\pi^{-1}(x)$.  See Section \ref{section:notation},    \eqref{eq:notation-conditional-meas:1}.  
We shall use the following properties about conditional measures of CP-distribution measures.

\begin{proposition}\label{prop:slice-cp-distribution:1}
Let $Q$ be an ergodic  CP-distribution on $\R^2$ and $\pi\in G(2,1)$.  Suppose $\dim \pi Q = \dim Q-\kappa$ for some $\kappa>0$. Then for $Q$-a.e.  $\mu$ and $\mu$-a.e.  $x$, the conditional measure $\mu_{\pi^{-1}(\pi(x))}$ is exact dimensional and $\dim \mu_{\pi^{-1}(\pi(x))}=\kappa$.  Moreover,  for any $\epsilon>0$, there exists $A$ with $Q(A)>1-\epsilon$ and $n_0,l_0\in \N$ such that for any $n\ge n_0$, for any $\mu\in A$, there exists $D$  with $\mu(D)>1-\epsilon$ and for each $x\in D$, the measure $\mu_{\pi^{-1}(\pi(x))}$ (viewed as a measure on $\R$) is $(n,l_0,\epsilon)$-dyadic spreading.
\end{proposition}
\begin{proof}
The first part of the proposition follows from Theorem 1.14 and Proposition 1.28 in \cite{Hochman2010-fractal-distributions}. Below, we provide a brief explanation of the key arguments. For detailed notation and definitions, we refer the reader to \cite{Hochman2010-fractal-distributions}.  By \cite[Theorem 1.14]{Hochman2010-fractal-distributions}, the continuous centering of $Q$, denoted P, is a ergodic fractal distribution. By \cite[Proposition 1.28]{Hochman2010-fractal-distributions}, the push-forward $P'$ of $P$ by the map $\mu\mapsto \left( \mu_{\pi^{-1}(0)} \right)^*$ is an ergodic fractal distribution, and the given map is a factor map between $(P,S^*)$ and $(P',S^*)$.   It follows that for $P$-a.e.  $\mu$ the conditional measure $\mu_{\pi^{-1}(0)}$ is well defined and $\mu_{\pi^{-1}(0)}$ generates $P'$ at $\mu_{\pi^{-1}(0)}$-a.e.  point.  The same holds for $\mu_{\pi^{-1}(x)}$,  for $P$-a.e. $\mu$ and $\pi\mu$-a.e. $x$. Thus for $P$-a.e.  $\mu$ and $\pi\mu$-a.e. $x$,  the measure $\mu_{\pi^{-1}(x)}$ generates $P'$.  Since the last property is invariant under translation, scaling and normalization, we conclude that for $Q$-a.e.  $\nu$, the measure $\nu$ also satisfies this property.     Since $P$ is ergodic and $(P',S^*)$  is a factor of $(P,S^*)$,  $P'$ is also ergodic.

We now present arguments for the dyadic spreading property of  $\mu_{\pi^{-1}(\pi(x))}$.  We have seen in the above that for $Q$-a.e.  $\mu$ and $\mu$-a.e.  $x$, the measure $\eta=\mu_{\pi^{-1}(\pi(x))}$ is well defined,  uniform scaling,  and  generates some ergodic fractal distribution $P'$ with $\dim P'>0$.  
Since $\dim P'>0$,  $P'$-a.e.  measure $\nu$ has positive dimension,  and in particular $\nu$ is non-atomic.  It follows that for any $\epsilon>0$ and $P'$-a.e.  $\nu$, there exists $\rho_\nu>0$ such that $\nu(I)<\epsilon$ for all interval $I$ of length $\rho_\nu$.  If we define $A_\rho$ to be the the set of measures $\nu'$ such that $\nu'(I)<\epsilon$ for every interval $I$ of length $\rho$, then for small enough $\rho>0$, the $P'$-measure of $A_\rho$ is at least $1-\epsilon$.  Since $\eta$ generates $P'$, for $\eta$-a.e.  $x$, we have $\frac{1}{T}\int_0^T\delta_{\{\eta^{x,t}\in A_\rho\}}\to P'(A_\rho)$ as $T\to \infty$.  It follows that for for $\eta$-a.e.  $x$ we have
\begin{equation}\label{eq:pf:prop:slice-cp-distribution:1}
\limsup_{T\to\infty} \frac{1}{T}\int_0^T \left( \sup \eta^{x,t}(I)  \right) dt <2\epsilon,
\end{equation}
where the $\sup$ is over all intervals $I$ of length $\le \rho.$
Note that here we view $\eta$ as a measure on $\R$.  

Now,  let us consider its translation $\eta_h=\eta+h$.   Then for Lebesgue almost every $h$,  $\eta_h$-a.e.  $x$ is $2$-normal, meaning that the sequence $(\{2^nx\})_{n}$ is uniformly distributed on $[0,1]$, where $\{y\}$ denotes the fractional part of $y$.  Let us fix such $h$.  Hence for $\eta_h$-a.e.  $x$,  we have
\[
\lim_{n\to\infty} \frac{1}{n}|\{ 1\le k\le n: \dist(\{2^n x \}, \frac{1}{2}) \ge 1-\frac{1}{2^l} \}| = \frac{1}{2^{l-1}}. 
\]
In particular, we have
\begin{equation}\label{eq:pf:prop:slice-cp-distribution:2}
\liminf_{n\to\infty}\frac{1}{n}|\{ 1\le k\le n:B(x,2^{-k-l})\subset \D_k(x) \}| \ge 1- \frac{1}{2^{l-1}}>1-\epsilon, 
\end{equation}
provided $l$ is large enough so that  $2^{1-l}<\epsilon$.
By \eqref{eq:pf:prop:slice-cp-distribution:1},  for $\eta_h$-a.e.  $x$, we have
\begin{equation*}
\limsup_{n\to\infty} \frac{1}{n}\sum_{k=1}^n \frac{\eta_h(B(x,2^{-k-l}))}{\eta_h(B(x,2^{-k})}  <4\epsilon,
\end{equation*}
provided $l$ is large enough in a way depending on $\rho$.  It follows that for $\eta_h$-a.e.  $x$, 
\begin{equation}\label{eq:pf:prop:slice-cp-distribution:3}
\limsup_{n\to\infty} \frac{1}{n}|\{ 1\le k\le n: \eta_h(B(x,2^{-k-l}) \le 2 \eta_h(B(x,2^{-k-2l})) \}|<8\epsilon.
\end{equation}
Combining \eqref{eq:pf:prop:slice-cp-distribution:2} and \eqref{eq:pf:prop:slice-cp-distribution:3}, and noticing that $\D_{k+2l+1}(x)\subset B(x,2^{-k-2l})$,  we get  for $\eta_h$-a.e.  $x$, 
\begin{equation*}\label{eq:pf:prop:slice-cp-distribution:4}
\limsup_{n\to\infty} \frac{1}{n}|\{ 1\le k\le n: \eta_h(\D_{k}(x)) \le 2 \eta_h(\D_{k+2l+1}(x)) \}|< 9\epsilon.
\end{equation*}
By Egorov's theorem, we obtain the desired conclusion.
\end{proof}

\subsection{An entropy increasing theorem}\label{subsection: entropy-increasing-theorem:1}

\begin{theorem}\label{thm:entropy-increase-1}
Given $0<\gamma<1$ and $l\in \N$, there exists $\delta>0$ such that the following holds for all $n$ large enough: Let $A\subset [0,1]\cap 2^{-n}\Z$,  $\eta\in \cP([0,1]\cap 2^{-n}\Z)$.  Suppose that $\eta(A)\ge 1/2$ and for each $x\in A$,
\begin{equation}\label{eq:thm:ent-incre:1}
\left| \left\{1\le k\le n: \eta(\D_k(x))\ge 2\eta(\D_{k+l}(x))\right\} \right| \ge (1-\gamma/2)n.
\end{equation}
Let $B\subset [0,1]\cap 2^{-n}\Z$,  and for each $a\in A$,  let $B_a\subset B$.   Suppose that $|B|\le 2^{n(1-\gamma)}$ and $|B_a|\ge |B|^{1-\delta}$ for all $a\in A$. Then we have
\[
\left| \bigcup_{a\in A}(a+B_a) \right| \ge |B|^{1+\delta}.
\]
\end{theorem}

Let us prepare some ingredients that we shall use in the proof of Theorem \ref{thm:entropy-increase-1}.
The conclusion of Theorem \ref{thm:entropy-increase-1} is a consequence of Hochman's inverse theorem (Theorem \ref{theorem:Hochman-inverse-thm}) and the Balog-Szemer\'edi-Gowers theorem. 
Let us recall the asymmetric Balog-Szemer\'edi-Gowers theorem. 

\begin{theorem}[Corollary 2.36 of \cite{TaoVu-book}]\label{thm:Balog-Szemeredi-Gowers}
For any $\epsilon>0$, there is $\delta>0$ such that the following holds for $L$ large enough.  Let $A, B$ be finite subsets of an additive group $Z$ and $G\subset A\times B$ such that $|A|\le L|B|$ and 
\[
|G|\ge L^{-\delta} |A|\cdot|B|
\] 
and 
\[
|\{a+b:(a,b)\in G\}|\le L^\delta|A|.
\]
Then there exist $A'\subset A, B'\subset B$ satisfying
\[
|A'|\ge L^{-\epsilon}|A|,  |B'|\ge L^{-\epsilon}|B|
\]
and
\[|A'+B'|\le L^{\epsilon}|A'|.\]
\end{theorem}
\begin{proof}
The assumptions $|G|\ge L^{-\delta} |A|\cdot|B|$ and $|\{a+b:(a,b)\in G\}|\le L^\delta|A|$ implies that we have
\[
|E(A,B)|\ge L^{-2\delta}|A||B|^2,
\]
where $E(A,B)=\{(a_1,a_2,b_1,b_2)\in A^2\times B^2: a_1+b_1=a_2+b_2\}.$ Now  \cite[Corollary 2.36]{TaoVu-book} implies the desired conclusion.
\end{proof}

We shall use the following regularization lemma for measures. It is a variant of Bourgain's regularization in \cite{Bourgain2010}.
The following version is proved in \cite[Lemma 3.4]{KS2019}.
\begin{lemma}\label{lemma:regularization measures}
Let $T\in \N$ be fixed.  For any $\mu\in \cP([0,1)^d)$ and any $1\le l\in \N$, there exist $A\subset \D_{lT}([0,1)^d)$ {\color{black} and sequence $(\sigma_1,\ldots,\sigma_l)\in [0,d]^l$} such that (writing $X=\cup_{Q\in A}Q$): (1) $\mu(X)\ge (2Td+2)^{-l}$; (2) for $1\le i\le l$,  any $Q\in  \D_{iT}([0,1)^d)$ with $\mu_{|X}(Q)>0$,   {\color{black} 
\[\mu_{|X}(Q)\le 2^{-\sigma_i T} \mu_{|X}(\widehat{Q})  \le 2\mu_{|X}(Q),\]
where $\widehat{Q}$ is the unique element in $ \D_{(i-1)T}([0,1)^d)$ containing $Q$.}
\end{lemma}

We shall also need the following lemma which is a standard porosity-type fact. 
\begin{lemma}\label{lemma:porous-set-small-mass}
Let $0<\tau<1,\gamma>0, 1\le l\in \N$ be given. Then there exist $\epsilon=\epsilon(\tau,\gamma,l)>0$  such that the following holds for all $n$  large enough.  Let $\mu\in \cP([0,1))$.  Let $A\subset [0,1)$ be such that
\begin{equation}\label{eq:lemma:porous-set-small-mass:1}
\left| \left\{1\le k\le n-l: \mu(\D_{k+l}(x)) \le \tau \mu(\D_{k}(x))\right\} \right| \ge n(1-\frac{\gamma}{2}) \ \textrm{ for } x\in A.
\end{equation}
Then for any set $D\subset A$ satisfying
 \begin{equation}\label{eq:lemma:porous-set-small-mass:2}
\left| \left\{1\le k\le n-l: {\color{black}\left| \left\{Q\in \D_{k+l}:Q\subset \D_k(x),  Q\cap D\neq\emptyset  \right\} \right|=1} \right\} \right| \ge n\gamma \ \textrm{ for } x\in D,
\end{equation}
we have
\[\mu(D)\le 2^{-n\epsilon}.\]
\end{lemma}
\begin{proof}
Assuming $\epsilon>0$ is small enough that we shall specify later,  we proceed to construct a probability measure $\nu$ on $[0,1)$ satisfying
\[
\nu(\D_n(x))>\mu(\D_n(x))2^{n\epsilon}  \   \textrm{ for } x\in D.
\]
This will immediate yield the desired conclusion.

We start with assigning the mass of $\nu$ on $\D_1([0,1))$.  Let $Q_0=[0,1)$ and 
\[a_{Q_0}=\left( \sum_{\substack{Q\in \D_1 \\ Q\cap D\neq \emptyset }}\mu(Q) \right)/\mu(Q_0).\]
For each $Q\in \D_1$ with $Q\cap D\neq \emptyset$, let
\[
\nu(Q)=\mu(Q)/a_{Q_0}.
\]
This defines a probability measure on $\D_1$.  
Let $1\le k\le n-1$. Suppose we have assigned the mass of $\nu$ on $\D_k$.  Let  $Q\in \D_k$ with $Q\cap D\neq \emptyset$.  Set
\[a_{Q}=\left( \sum_{\substack{Q'\in \D_{k+1} \\ Q'\cap D\neq \emptyset }}\mu(Q') \right)/\mu(Q).\]
For each $Q'\in \D_{k+1}$ with $Q'\cap D\neq \emptyset$, let
\[
\nu(Q')=\nu(Q) \frac{\mu(Q')}{\mu(Q)a_{Q_0}}.
\]
In such way,  we  construct a probability measure on  $\D_n$.  By definition, for each $x\in D$, we have
\[
\nu(\D_n(x))=\prod_{k=0}^{n-1}\frac{\mu(\D_{k+1}(x))}{\mu(\D_{k}(x))a_{\D_{k}(x)}}=\mu(\D_n(x))\prod_{k=0}^{n-1}\frac{1}{a_{\D_{k}(x)}}.
\]
{\color{black}For notational convenience,  we shall write $\left|\D_{k+l}\cap \D_k(x)\cap D \right|=1$ to mean  \[\left| \left\{Q\in \D_{k+l}:Q\subset \D_k(x),  Q\cap D\neq\emptyset  \right\} \right|=1.\] }
By \eqref{eq:lemma:porous-set-small-mass:1} and \eqref{eq:lemma:porous-set-small-mass:2}, for each $x\in D$ we have
 \begin{equation}\label{eq:proof:lemma:porous-set-small-mass:1}
\left| \left\{1\le k\le n-l: \mu(\D_{k+l}(x)) \le \tau \mu(\D_{k}(x)) \textrm{ and } \left|\D_{k+l}\cap \D_k(x)\cap D \right|=1  \right\} \right| \ge n(\gamma-\gamma/2).
\end{equation}
Observe that if {\color{black}$\left|\D_{k+l}\cap \D_k(x)\cap D \right|=1$,}  we must have 
\[
a_{\D_{k+j}(x)}=\frac{\mu(\D_{k+j+1}(x))}{\mu(\D_{k+j}(x))} \textrm{ for } 0\le j\le l-1,
\]
which implies that 
\begin{equation}\label{eq:proof:lemma:porous-set-small-mass:3}
a_{\D_{k}(x)}\ldots a_{\D_{k+l-1}(x)}=\frac{\mu(\D_{k+l}(x))}{\mu(\D_{k}(x))}. 
\end{equation}
By \eqref{eq:proof:lemma:porous-set-small-mass:1} and \eqref{eq:proof:lemma:porous-set-small-mass:3},  it follows that
\[
\prod_{k=0}^{n-1}\frac{1}{a_{\D_{k}(x)}}\ge \tau^{-n\rho}
\]
for some $\rho>0$ depending on  $\gamma/2$ and $l$.  Thus for each $x\in D$, we have
\[
\nu(\D_n(x))\ge \mu(\D_n(x))\tau^{-n\rho}.
\]
Therefore
\[
\mu(D)\le \sum_{\substack{Q\in \D_n \\ Q\cap D\neq \emptyset}}\mu(Q)\le \sum_{\substack{Q\in \D_n \\ Q\cap D\neq \emptyset}}\nu(Q) \tau^{n\rho} \le  \tau^{n\rho}\le  2^{-n\epsilon}
\]
provided $\epsilon$ is small enough in terms of $\tau$ and $\rho$.
\end{proof}

Now we are ready to prove Theorem \ref{thm:entropy-increase-1}.  
\begin{proof}[Proof of Theorem \ref{thm:entropy-increase-1}]
Let $\eta$ and $A$ be as in Theorem \ref{thm:entropy-increase-1}.  Let us fix a $T\in \N$.  We suppose that $T$ is large enough in a way depending of $\gamma$ and $l$ that we shall specify later. 

For $j\in \N$,  let $A(j)=\{a\in A: \eta(\D_n(a))\in [2^{-nj/T},2^{-n(j+1)/T}) \}$.  Let us pick $j_0\in \{0,\ldots, T\}$ such that 
\[
\sum_{a\in A(j_0)}\eta(\D_n(a))=\max_{0\le j\le T}\sum_{a\in A(j)}\eta(\D_n(a)).
\]
Observe that we have
\[
\sum_{j\ge T+1}\sum_{a\in A(j)}\eta(\D_n(a))\le 2^n 2^{-n(1+1/T)} =2^{-n/T}.
\]
Thus we have
\begin{equation}\label{eq:pf-Thm-ent-inc:1}
\sum_{a\in A(j_0)}\eta(\D_n(a))\ge (1-2^{-n/T})/(T+1).
\end{equation}
In the following, we shall simply write $A_0$ for $A(j_0)$. 
By definition of $A_0$ we have
\begin{equation}\label{eq:pf-Thm-ent-inc:2}
2^{-n/T}\le \frac{\eta(\D_n(a))}{\eta(\D_n(a'))}\le 2^{n/T} \textrm{ for } a,a'\in A_0.
\end{equation}

Now, let $B\subset [0,1]\cap 2^{-n}\Z$ and $B_a\subset B, a\in A_0$ be such that  
\begin{equation}\label{eq:pf-Thm-ent-inc:3}
|B_a|\ge |B|^{1-\epsilon} \textrm{ for all } a\in A_0,
\end{equation}
and
\begin{equation}\label{eq:pf-Thm-ent-inc:4}
\left|\bigcup_{a\in A_0}(a+B_a)\right|\le |B|^{1+\epsilon}.
\end{equation}
Our goal is to show that \eqref{eq:pf-Thm-ent-inc:3} and \eqref{eq:pf-Thm-ent-inc:4} could not hold simultaneously if $\epsilon$ is small enough in a way depending on $\gamma$ and $l$.

Let 
\[G=\bigcup_{a\in A_0}\{a\times B_a\}\subset A_0\times B.\]
Note that we have $|B|\le 2^n|A|$.  Thus we may apply the asymmetric Balog-Szemer\'edi-Gowers theorem (Theorem \ref{thm:Balog-Szemeredi-Gowers}) to conclude that there exist $\delta=o_\epsilon(1)$,  explicitly depending on $\epsilon$, and $A'\subset A_0$, $B'\subset B$ such that 
\begin{equation}\label{eq:pf-Thm-ent-inc:5}
|A'|\ge  2^{-n\delta}|A_0|,
\end{equation}
\begin{equation}\label{eq:pf-Thm-ent-inc:6}
|B'|\ge  2^{-n\delta}|B|,
\end{equation}
\begin{equation}\label{eq:pf-Thm-ent-inc:7}
|A'+B'|\le  2^{n\delta}|B'|.
\end{equation}
Let 
\begin{equation}\label{eq:pf-Thm-ent-inc:8-bis}
\nu_1=\frac{1}{|A'|}\sum_{a\in A'}\delta_{a} \  \textrm{ and } \  \nu_2=\frac{1}{|B'|}\sum_{b\in B'}\delta_{b}.
\end{equation}
Then $\nu_1$ and $\nu_2$ are probability measures on $A'$ and $B'$,  respectively.  Observe that we have
\begin{equation}\label{eq:pf-Thm-ent-inc:8}
H(\nu_2,\D_n)=\log |B'|.
\end{equation}
Write $m=\lfloor n/T\rfloor$.  
{\color{black}By Lemma \ref{lemma:regularization measures},   there exist $A''\subset A'$ and  $(\sigma_1,\ldots,\sigma_m)\in [0,d]^m$  such that 
\begin{equation}\label{eq:pf-Thm-ent-inc:9}
\nu_1(A'') > (2T+2)^{-m}
\end{equation}
and for $1\le i\le m$,  any $Q\in  \D_{iT}([0,1)^d)$ with ${\nu_1}_{|A''}(Q)>0$,   
\begin{equation}\label{eq:pf-Thm-ent-inc:10}
{\nu_1}_{|A''}(Q)\le 2^{-\sigma_i T} {\nu_1}_{|A''}(\widehat{Q})  \le 2{\nu_1}_{|A''}(Q),
\end{equation}
where $\widehat{Q}$ is the unique element in $ \D_{(i-1)T}([0,1)^d)$ containing $Q$.}
Let \[\tilde{\nu}_1=\frac{{\nu_1}_{|A''}}{\nu_1(A'')}.\]
Observe that since $T$ is large, we have
\begin{equation}\label{eq:pf-Thm-ent-inc:11}
(2T+2)^{-m}=2^{-mT\log (2T+2)/T} = 2^{-n \cdot o_T(1)}.
\end{equation}
Note that we have
\[
H(\tilde{\nu}_1*\nu_2,\D_n)\le \log |A''+B'| \le \log |A'+B'|\le \log|B'|+n\delta=H(\nu_2,\D_n)+n\delta.
\]
Applying Hochman's inverse theorem (Theorem \ref{theorem:Hochman-inverse-thm}),  we get that whenever $\delta$ is small enough,  there exist small constant $\tau=o_{\delta}(1)$ and $I,J\subset \{1,\ldots,n\}$ such that $|I\cup J|>(1-\tau)n$ and 
\begin{equation}\label{eq:pf-Thm-ent-inc:12}
\nu_2\left(x:H(\nu_2^{\D_i(x)},\D_T)>(1-\tau)T\right) >1-\tau \textrm{ for } i\in I,
\end{equation}
\begin{equation}\label{eq:pf-Thm-ent-inc:13}
\tilde{\nu}_1\left(x: \tilde{\nu}_1^{\D_j(x)}(Q)>(1-\tau) \textrm{ for some } Q\in \D_{2T} \right) >1-\tau \textrm{ for } j\in J.
\end{equation}
Recall that by \eqref{eq:pf-Thm-ent-inc:8} we have
\begin{equation}\label{eq:pf-Thm-ent-inc:14-1}
H(\nu_2,\D_n)=\log |B'|\le \log |B|\le (1-\gamma)n. 
\end{equation}
We now need  the following formula  which holds for any measure $\nu\in \cP([0,1])$ (see \cite[Lemma 3.4]{Hochman2014}):
\begin{equation}\label{eq:pf-Thm-ent-inc:14-2}
H(\nu,\D_n)=\sum_{k=0}^{n-1}\sum_{Q\in \D_k}\nu(Q)\frac{H(\nu^Q,\D_T)}{T} +O(T).
\end{equation}
Combining \eqref{eq:pf-Thm-ent-inc:12},  \eqref{eq:pf-Thm-ent-inc:14-1} and \eqref{eq:pf-Thm-ent-inc:14-2}, it follows that 
we must have
\[
|I|\le (1-\gamma+o_\tau(1))n.
\]
Thus we have
\begin{equation}\label{eq:pf-Thm-ent-inc:15}
|J| \ge (\gamma-o_\tau(1))n.
\end{equation}

{\color{black}
For $j\in J$,  let $F_j$ be the set of $x$ such that 
\begin{equation}\label{eq:pf-Thm-ent-inc:16*}
\tilde{\nu}_1^{\D_j(x)}(Q)>(1-\tau) \textrm{ for some } Q\in \D_{2T} .
\end{equation}
Then by \eqref{eq:pf-Thm-ent-inc:13} we have $\tilde{\nu}_1(F_j)>1-\tau.$  

Let us fix $j\in J$ and $x\in F_j$. There exists a unique $k$ such that $j<kT\le j+T$.  Hence $(k+1)T\in (j+T,j+2T]$.  By Property \eqref{eq:pf-Thm-ent-inc:10},   it follows that for each $Q\in \D_{(k-1)T}$ and all $Q_1,Q_2\in \D_{(k+1)T}$ such that $Q_1,Q_2\subset Q$ and $\tilde{\nu}_1(Q_1)>0, \tilde{\nu}_1(Q_2)>0$,  we have 
\begin{equation}\label{eq:pf-Thm-ent-inc:16.1}
\frac{1}{4}\le \frac{\tilde{\nu}_1(Q_1)}{\tilde{\nu}_1(Q_2)} \le 4.
\end{equation}
Since $j\in J$ and $x\in F_j$,  by property \eqref{eq:pf-Thm-ent-inc:16*}, there exists $Q_0\in \D_{j+2T}$ such that $Q_0\subset \D_j(x) $ and 
\begin{equation}\label{eq:pf-Thm-ent-inc:16.2}
\tilde{\nu}_1(Q_0)>(1-\tau)\tilde{\nu}_1(\D_j(x)).
\end{equation}
We claim that by \eqref{eq:pf-Thm-ent-inc:16.1} and \eqref{eq:pf-Thm-ent-inc:16.2}, we must have 
\begin{equation}\label{eq:pf-Thm-ent-inc:16**}
\left| \left\{ Q'\in \D_{(k+1)T}: Q'\subset  D_j(x) \textrm{  and  }   \tilde{\nu}_1(Q')>0  \right\} \right| \le \frac{4}{1-\tau}.
\end{equation}
For otherwise, there exists $ Q'\in \D_{(k+1)T}$ such that $Q'\subset  D_j(x) $ and 
\begin{equation}\label{eq:pf-Thm-ent-inc:16.3}
0<\tilde{\nu}_1(Q')< \frac{(1-\tau)}{4}\tilde{\nu}_1(\D_j(x)).
\end{equation}
Recall that $Q_0\in \D_{j+2T}$ and $j+2T\ge (k+1)T$.   Let $\widehat{Q}_0\in \D_{(k+1)T}$ be the unique element containing $Q_0$. Then by \eqref{eq:pf-Thm-ent-inc:16.2} we have 
\begin{equation}\label{eq:pf-Thm-ent-inc:16.4}
 \tilde{\nu}_1(\widehat{Q}_0) \ge \tilde{\nu}_1(Q_0)\ge (1-\tau) \tilde{\nu}_1(\D_j(x)).
\end{equation}
Combining \eqref{eq:pf-Thm-ent-inc:16.3} and \eqref{eq:pf-Thm-ent-inc:16.4}, we get 
\[
\tilde{\nu}_1(\widehat{Q}_0) > (1-\tau)\frac{4}{(1-\tau)}\tilde{\nu}_1(Q')=4\tilde{\nu}_1(Q').
\]
This is a contradiction to \eqref{eq:pf-Thm-ent-inc:16.1} as we have $\widehat{Q}_0,  Q'\subset  \D_{(k-1)T}(x)$, $\widehat{Q}_0,  Q'\in  \D_{(k+1)T}$  and $\tilde{\nu}_1(\widehat{Q}_0) >0,  \tilde{\nu}_1(Q')>0$.  Thus we must have \eqref{eq:pf-Thm-ent-inc:16**}.

Recall that $\tau$ has been chosen small enough, in particular, we may assume that $\tau<1/2$,   so that $\frac{4}{1-\tau}<8$.

Note that by property \eqref{eq:pf-Thm-ent-inc:10}, if $Q'\in \D_{(k+1)T}$, then $\tilde{\nu}_1(Q')>0$ if and only if $Q'\cap A''\neq \emptyset.$ Thus \eqref{eq:pf-Thm-ent-inc:16**} becomes 
\begin{equation}\label{eq:pf-Thm-ent-inc:16**.1}
\left| \left\{ Q'\in \D_{(k+1)T}: Q'\subset  D_j(x) \textrm{  and  }  Q'\cap A''\neq \emptyset  \right\} \right| \le \frac{4}{1-\tau}<8.
\end{equation}
It is not hard to see that the above property implies the following 
\begin{equation}\label{eq:pf-Thm-ent-inc:16***}
\left|\left\{j< i\le (k+1)T: \left| \left\{ Q'\in \D_{i}: Q'\subset  D_{i-1}(x) \textrm{  and  }  Q'\cap A''\neq \emptyset  \right\} \right|>1\right\}\right| \le 8.
\end{equation}
Recall that $j<kT$.   Thus we have shown that for all $j\in J$ and $x\in F_j$, the property \eqref{eq:pf-Thm-ent-inc:16***} holds.

Recall that $\tilde{\nu}_1(F_j)>1-\tau$ for each $j\in J$.  By Fubini's theorem and Markov's inequality, we can find a set $F$ with   $\tilde{\nu}_1(F)>1-\sqrt{\tau}$ and for each $x\in F$, we have
\begin{equation}\label{eq:pf-Thm-ent-inc:16****}
\left|\left\{ j\in J: x\in F_j  \right\}\right| \ge |J|(1-\sqrt{\tau}).
\end{equation}
Now,  let us fix $x\in F$. Let $J_x=\{j\in J: x\in F_x\}$.
Then we have 
\[
|J_x| \ge |J|(1-\sqrt{\tau})\ge \gamma(1-o_\tau(1))n.
\]
Recall that $m=\lfloor n/T\rfloor$.  Since $|J_x|\ge \gamma(1-o_\tau(1))n$,  we must have 
\begin{equation}\label{eq:pf-Thm-ent-inc:16.5-1}
\left| \left\{ 1\le k\le m: J_x\cap[(k-1)T,kT )\neq \emptyset  \right\} \right| \ge  \gamma(1-o_\tau(1)-o_T(1))n.
\end{equation}
Note that for each $j\in J_x$, we have \eqref{eq:pf-Thm-ent-inc:16***}.  Hence for each $j\in J_x$,  with $j\in [(k-1)T,kT )$, we have 
\begin{equation}\label{eq:pf-Thm-ent-inc:16.6}
\left|\left\{kT\le i\le (k+1)T: \left| \left\{ Q\in \D_{i}: Q\subset  D_{i-1}(x) \textrm{  and  }  Q\cap A''\neq \emptyset  \right\} \right|>1\right\}\right| \le 8.
\end{equation}
Recall that $T$ is large compared to $l$. Thus it follows from \eqref{eq:pf-Thm-ent-inc:16.6} that we have 
\begin{equation}\label{eq:pf-Thm-ent-inc:16.7}
\left|\left\{kT\le i\le (k+1)T: \left| \left\{ Q\in \D_{i+l}: Q\subset  D_{i}(x) \textrm{  and  }  Q\cap A''\neq \emptyset  \right\} \right|=1\right\}\right| \ge T(1-o_T(1)).
\end{equation}
Thus, combining \eqref{eq:pf-Thm-ent-inc:16.5-1} and \eqref{eq:pf-Thm-ent-inc:16.7},  we obtain 
\begin{equation*}
\left|\left\{1\le i\le n: \left| \left\{ Q\in \D_{i+l}: Q\subset  D_{i}(x) \textrm{  and  }  Q\cap A''\neq \emptyset  \right\} \right|=1\right\}\right| \ge \gamma(1-o_\tau(1)-o_T(1))n \ge \frac{\gamma}{2}n.
\end{equation*}
The above property holds true for all $x\in F$.

Recall that we have assumed that $\eta$ satisfies \eqref{eq:thm:ent-incre:1}.  Thus we may apply Lemma \ref{lemma:porous-set-small-mass} to conclude that 
\begin{equation}\label{eq:pf-Thm-ent-inc:19}
\eta(F)\le 2^{-n\rho}
\end{equation}
for some absolute constant $\rho$ depending only on $\gamma$ and $l$.

We recall that $F\subset A''$ and $\tilde{\nu}_1(F)=\nu_1(F)/\nu_1(A'')\ge 1-\sqrt{\tau}$.  }
We also recall that (see \eqref{eq:pf-Thm-ent-inc:1}, \eqref{eq:pf-Thm-ent-inc:2}, \eqref{eq:pf-Thm-ent-inc:5},  \eqref{eq:pf-Thm-ent-inc:8-bis}, \eqref{eq:pf-Thm-ent-inc:8} and \eqref{eq:pf-Thm-ent-inc:9}):
\begin{eqnarray*}
 & & A''\subset A'\subset A_0\subset A,  \\
& & \eta(A)\ge 1/2 \textrm{ and } \eta(A_0)\ge (1-2^{-n/T})/(T+1), \\
& & 2^{-n/T}\le \frac{\eta(a)}{\eta(a')}\le 2^{n/T} \textrm{ for } a,a'\in A_0,\\
& & |A'|\ge 2^{-n\delta}|A_0|, \\
& & |A''| \ge (2T+2)^{-m}|A'|. 
\end{eqnarray*}
From these estimates, we get 
\[
\eta(F)\ge (1-\sqrt{\tau})\cdot(2T+2)^{-m}\cdot 2^{-n\delta} \cdot 2^{-2n/T}  \cdot \frac{1}{2}  \cdot (1-2^{-n/T})/(T+1) >2^{-n\rho}
\]
 provided $T$ is large enough and $\delta$ is small enough in terms of $\rho$. This is a contradiction to \eqref{eq:pf-Thm-ent-inc:19}.
\end{proof}

\subsection{Proof of Theorem \ref{thm:proj-thm-cp-distribution:-2}}\label{subsection: proof-of-proj-theorem-cp-dist:1}

Now we are ready to prove Theorem  \ref{thm:proj-thm-cp-distribution:-2}.
\begin{theorem}\label{thm:proj-cp-distribution:1}
Let $Q$ be an ergodic CP-distribution on $\R^2$. Then the following set is at most countable:
\begin{equation}\label{eq:thm:proj-cp-distribution:1}
\{ \pi\in S^1: \dim\pi Q < \min(1,\dim Q)\}.
\end{equation}
\end{theorem}
We shall use the following easy fact in the proof of Theorem \ref{thm:proj-cp-distribution:1}.
\begin{lemma}\label{lemma:entropy of restricted meas-small delete}
Let $\mu\in \cP([0,1]^d)$. Suppose that $A\subset [0,1]^d$ such that $\mu(A)>1-\epsilon$. Then for any partition $\mathcal{A}$ of $[0,1]^d$,
\[
H(\mu_{|A},\mathcal{A}) >  H(\mu,\mathcal{A})-c\sqrt{\epsilon}\log |\mathcal{A}|,
\]
where $c$ is a constant depending only on $\R^d$.
\end{lemma}

\begin{proof}[Proof of Theorem \ref{thm:proj-cp-distribution:1}]
Let us suppose,  by contradiction, that the exceptional set \eqref{eq:thm:proj-cp-distribution:1}  is uncountable.  It is not difficult to check that the set \eqref{eq:thm:proj-cp-distribution:1} is a Borel set.  There exists $\kappa>0$ such that the set 
\[
E=\{ \pi\in G(2,1): \dim \pi Q<\min(1,\dim Q)-\kappa \}
\]
is uncountable and there exists a non-atomic Borel probability measure $\nu$ that is supported on a compact subset of $E$.

Applying Proposition \ref{prop:slice-cp-distribution:1} to $Q$ and $\pi\in E$ and by Egorov's theorem,  we obtain that for any $\epsilon_0>0$, there exist $l_0,  n_0\in \N$,  and $E_0\subset E$ with $\nu(E_0)>1-\epsilon_0$, and for each $\pi\in E_0$, there exists $A_\pi^0$ such that $Q(A_\pi^0)>1-\epsilon_0$ and for each $\mu\in A_\pi^0$, and $n\ge n_0$  there exists $D_{\pi,\mu}^0$ with $\mu(D_{\pi,\mu}^0)>1-\epsilon_0$ and for each $x\in D_{\pi,\mu}^0$, the measure $\mu_{\pi^{-1}(\pi(x))}$ is well defined  and $(n,l_0,\epsilon_0)$-dyadic spreading.

Applying Proposition \ref{prop:dim-projection-cp-distribution:1} to $Q$ and $\pi\in E$ and by Egorov's theorem,  we obtain that for any $\epsilon_1>0$,  there exists $l_1\in \N$ and $E_1\subset E$ with $\nu(E_1)>1-\epsilon_1$ such that for all  $\pi\in E_1,$ and  $l\ge l_1$, there exists $A_\pi^1$ such that $Q(A_\pi^1)>1-\epsilon_1$ and for any $\mu\in A_\pi^1$, there exists $D_{\pi,\mu}^1$ with $\mu(D_{\pi,\mu}^1)>1-\epsilon_1$ and for each $x\in D_{\pi,\mu}^1$ we have (writing $\beta_\pi=\dim \pi Q$)
\begin{eqnarray}
 & & \mu(B(x,r))= r^{\alpha\pm \epsilon_1} \  \textrm{ for }  r\le 2^{-l_1},  \label{eq:thm-proj-cp-distribution-1:2}\\
& & \pi\mu(B(\pi(x),r)) = r^{\beta_\pi\pm \epsilon_1} \  \textrm{ for } r\le 2^{-l_1}, \label{eq:thm-proj-cp-distribution-1:3}\\
& & \frac{1}{l}H\left(\mu^{\D_l(x)},\D_l\right) = \alpha\pm \epsilon_1,\label{eq:thm-proj-cp-distribution-1:4}\\
& & \frac{1}{l}H\left(\pi\left(\mu^{\D_l(x)}\right),\D_l\right) = \beta_\pi\pm \epsilon_1. \label{eq:thm-proj-cp-distribution-1:5}
\end{eqnarray}
Note that by \eqref{eq:thm-proj-cp-distribution-1:3},  for each $\pi\in E_1$ and $\mu\in A_{\pi}^1$, we necessarily have
\begin{equation}\label{eq:thm-proj-cp-distribution-1:6}
N_{r}(\pi(D_{\pi,\mu}^1)) = r^{-\beta_\pi\pm \epsilon_1} \ \textrm{ for } r\le 2^{-l_1}.
\end{equation}
Note that we have $\nu(E_0\cap E_1)>1-\epsilon_0-\epsilon_1>0$, provided $\epsilon_0$ and $\epsilon_1$ were assumed small enough.  Since $\nu$ is non-atomic, the set $E_0\cap E_1$ is uncountable. In particular,  for any $\epsilon_2>0$,   there must exist $E_2\subset E_0\cap E_1$ with $\nu(E_2)>0$ such that 
 \begin{equation}\label{eq:thm-proj-cp-distribution-1:7}
|\dim \pi Q-\dim \pi' Q | <\frac{\epsilon_2}{2} \ \textrm{ for }\pi,\pi'\in E_2.
\end{equation}

In the following, we shall show that there is a contradiction if $\epsilon_0,\epsilon_1$ and $\epsilon_2$ were assumed  small enough in terms of $\beta$ and $l_0$. 

Let us pick some $\pi_0\in E_2$ and denote  $\beta_0=\pi_0 Q$.  Then for each $\pi\in E_2$, we have 
\[
\dim \pi Q = \beta_0\pm \frac{\epsilon_2}{2} .
\]
Let $\rho_0=2^{-l_1-1}$.   
 Since $E_2$ is infinite (uncountable), there exist $\pi_1,\pi_2\in E_2$ such that 
\begin{equation}\label{eq:thm-proj-cp-distribution-1:8}
0<\dist(\pi_1,  \pi_2)\le \rho_0. 
\end{equation}
Then  we have 
\[
l:=\lfloor -\log \dist(\pi_1,\pi_2) \rfloor \ge \lfloor -\log \rho_0 \rfloor \ge l_1+1.
\]
Recall that the logarithm $\log$ is in base 2. Note that we have $\dist (\pi_1,\pi_2)=2^{-l\pm 1}$.  We have seen that (consequence of Proposition \ref{prop:dim-projection-cp-distribution:1}),  for each $\pi\in \{\pi_1,\pi_2\}$, there exists $A_{\pi}^1$ with $Q(A_\pi^1)>1-\epsilon_1$ and for each $\mu\in A_\pi^1$, there exists $D_{\pi,\mu}^1$ with $\mu(D_{\pi,\mu}^1)>1-\epsilon_1$ and for each $x\in D_{\pi,\mu}^1$ we have  \eqref{eq:thm-proj-cp-distribution-1:2},  \eqref{eq:thm-proj-cp-distribution-1:3}, \eqref{eq:thm-proj-cp-distribution-1:4}, and \eqref{eq:thm-proj-cp-distribution-1:5}.
Let 
\[
A=A_{\pi_1}^0 \cap A_{\pi_1}^1 \cap A_{\pi_2}^0 \cap A_{\pi_2}^1. 
\]
Then we have $Q(A)\ge 1-2(\epsilon_0+\epsilon_1)$.   For $\mu\in A$, let 
\[
D_\mu=D_{\pi_1,\mu}^0 \cap D_{\pi_1,\mu}^1 \cap D_{\pi_2,\mu}^0 \cap D_{\pi_2,\mu}^1.
\]

Recall that we denoted $\beta_\pi=\dim \pi Q$,   and   $\beta_0=\dim \pi_0 Q$ for some $\pi_0\in E_2$.  By \eqref{eq:thm-proj-cp-distribution-1:6} and \eqref{eq:thm-proj-cp-distribution-1:7},  we have
\begin{equation}\label{eq:thm-proj-cp-distribution-1:9}
N_{r}(\pi(D_\mu)) \le r^{-\beta_0-\epsilon_2- \epsilon_1} \ \textrm{ for } \pi\in\{\pi_1,\pi_2\},r\le 2^{-l_1}.
\end{equation}
Note that we have  $\mu(D_\mu)>1-2(\epsilon_0+\epsilon_1)$.
By Markov's inequality,  there exists $\epsilon_3=o_{\epsilon_0,\epsilon_1}(1)$ (we may take $\epsilon_3=\sqrt{2(\epsilon_0+\epsilon_1)}$) and $\tilde{D}_\mu\subset D_\mu$ with $\mu(\tilde{D}_\mu)>1-\epsilon_3$ such that for each $x\in \tilde{D}_\mu$,   we have 
\begin{eqnarray}
 & & \mu_{\pi_1^{-1}(\pi_1(x))} \textrm{ is well defined and }   \mu_{\pi_1^{-1}(\pi_1(x))}(\tilde{D}_\mu) >1-\epsilon_3, \label{eq:thm-proj-cp-distribution-1:i}\\
& & \frac{\mu(\D_l(x)\cap \tilde{D}_\mu)}{\mu(\D_l(x))} > 1-\epsilon_3. \label{eq:thm-proj-cp-distribution-1:iii}
\end{eqnarray}

Since $\pi_1\mu(\pi_1\tilde{D}_\mu)\ge \mu(\tilde{D}_\mu)>1-\epsilon_3$ and for each $x\in \tilde{D}_\mu$ we have (recall \eqref{eq:thm-proj-cp-distribution-1:3}) 
\[
\pi_1\mu(B(\pi_1(x),  r)) \le r^{\beta_0-\epsilon_2-\epsilon_1} \  \textrm{ for } r\le 2^{-l_1}, 
\]
we must have
\begin{equation}\label{eq:thm-proj-cp-distribution-1:10}
N_{r}(\pi_1(\tilde{D}_\mu)) \ge c \cdot  r^{-\beta_0+\epsilon_2+ \epsilon_1} \ \textrm{ for } r\le 2^{-l_1}
\end{equation}
for some absolute constant $c$ depending only on $\R$.  Specific to $r=2^{-l}$,  it follows from \eqref{eq:thm-proj-cp-distribution-1:10} that we can find points $z_i\in \pi_1(\tilde{D}_\mu),  1\le i\le \lfloor\frac{1}{16}\cdot c \cdot  2^{l(\beta_0-\epsilon_2- \epsilon_1)}  \rfloor$, such that 
\begin{equation}\label{eq:thm-proj-cp-distribution-1:10.1}
\dist(z_i,z_j)\ge 2^{-l+3} \ \textrm{ for } i\neq  j.
\end{equation}
Let us now fix any $i\in \{1,\ldots,  \lfloor\frac{1}{16}\cdot c \cdot  2^{l(\beta_0-\epsilon_2- \epsilon_1)}  \rfloor\}$.  
In the following, we are going to show that 
\begin{equation}\label{eq:thm-proj-cp-distribution-1:10-1}
N_{2^{-2l}}\left( \pi_2\left( \pi_1^{-1}(B(z_i,2^{-l+1})) \right) \right) \ge 2^{l(\beta_0+\delta/2)}
\end{equation}
for some $\delta>0$ depending only on $l_0$ and $\beta$.
Note that since  $\dist(\pi_1,\pi_2)= 2^{-l\pm 1}$,  in view of \eqref{eq:thm-proj-cp-distribution-1:10.1}, we infer that  whenever $i\neq j$, we have 
\[
 \pi_2\left( \pi_1^{-1}(B(z_i,2^{-l+1})) \right) \cap  \pi_2\left( \pi_1^{-1}(B(z_j,2^{-l+1})) \right) = \emptyset.
\]
Hence, once \eqref{eq:thm-proj-cp-distribution-1:10-1} is proved,  we then can deduce that 
\[
N_{2^{-2l}}\left( \pi_2(\tilde{D}_\mu)\right) \ge \sum_{i=1}^{\lfloor\frac{1}{16}\cdot c \cdot  2^{l(\beta_0-\epsilon_2- \epsilon_1)}  \rfloor} 2^{l(\beta_0+\delta/2)} \ge 2^{2l(\beta_0+\delta/4)} >  2^{2l(\beta_0+2(\epsilon_2+\epsilon_1))},
\]
provided $\epsilon_2 $ and $\epsilon_1$ were assumed small enough.  This is a contradiction to \eqref{eq:thm-proj-cp-distribution-1:9}, thus concluding the proof of Theorem \ref{thm:proj-cp-distribution:1}.

Let us proceed to the proof of \eqref{eq:thm-proj-cp-distribution-1:10-1}. Let 
\[L'=\pi_1^{-1}(z_i)\cap \tilde{D}_\mu\ \textrm{ and } \ A=\{a\in \D_l(\R^2): a\cap L'\neq \emptyset\}.\]
For each $a\in A$, let us fix $x_a\in a\cap L'$.   For each $a\in A$, let $B_a'=a\cap \tilde{D}_\mu$.  Finally, let 
\[B_a=\pi_1 (B'_a) \ \textrm{ and } \  B=\pi_1(\cup_{a\in A}B'_a)=\cup_{a\in A}B_a.\]
Observe that the set $B$ has diameter bounded by $4\cdot 2^{-l}$.  By  \eqref{eq:thm-proj-cp-distribution-1:3},  for each $y\in B$ we have
\[
\pi_1\mu(B(y,  r)) = r^{\beta_0\pm\epsilon_2\pm\epsilon_1} \  \textrm{ for } r\le 2^{-l_1}.
\]
Hence we must have 
\begin{equation}\label{eq:thm-proj-cp-distribution-1:11}
N_{2^{-2l}}(B) \le  c \cdot  2^{l(\beta_0+2\epsilon_2+2 \epsilon_1)},
\end{equation}
for some absolute constant $c$ depending only on $4$ and $\R$.  

Recall that for each $x\in \tilde{D}_\mu$,  the property \eqref{eq:thm-proj-cp-distribution-1:5} is satisfied.  Combining this with  \eqref{eq:thm-proj-cp-distribution-1:iii} and Lemma \ref{lemma:entropy of restricted meas-small delete},  we obtain that the following holds for each $a\in A$:
\[
\frac{1}{l}H\left(\pi_1\left(\left(\mu_{|\tilde{D}_\mu}\right)^{\D_l(x_a)}\right),  \D_l\right)\ge \beta_0-\epsilon_2- \epsilon_1-o_{\epsilon_3}(1). 
\]
Since the measure $\pi_1\left(\left(\mu_{|\tilde{D}_\mu}\right)^{\D_l(x_a)}\right)$ is supported on $\pi_1(B_a')=B_a$, we have
\begin{equation}\label{eq:thm-proj-cp-distribution-1:12}
N_{2^{-2l}}(B_a) \ge   c \cdot  2^{l(\beta_0-\epsilon_2- \epsilon_1-o_{\epsilon_3}(1))}
\end{equation}
for some absolute constant $c>0$ depending only on $\R$.

Now, let us consider the following set 
\[
\pi_2\left( \bigcup_{a\in A}B_a' \right).
\]
Since $\dist(\pi_1,\pi_2)\le 2^{-l+1}$ and the diameter of $B_a'$ is bounded by $2^{-l+1}$,  we have
\begin{equation}\label{eq:thm-proj-cp-distribution-1:13}
d_{\rm H}(\pi_2(B_a'),  \pi_1(B_a')+ \pi_2(x_a) -\pi_1(x_a)) \le  c\cdot 2^{-2l},
\end{equation}
where $c>0$ is some absolute constant only depending on $\R^2$.  Here $d_{\rm H}$ stands for the Hausdorff distance between subsets of $\R^1$. 
Recall that for each $a\in A$ we have
\begin{equation}\label{eq:thm-proj-cp-distribution-1:14}
\pi_1(x_a)=z_i.
\end{equation}
Hence by \eqref{eq:thm-proj-cp-distribution-1:13} we have
{\color{black}
\[
d_{\rm H}\left(\pi_2\left(\bigcup_{a\in A}B_a'\right),  \bigcup_{a\in A}\left(\pi_1(B_a')+ \pi_2(x_a) -z_i\right)\right)\le  c\cdot 2^{-2l}
\]}
Recall that we denoted $B_a=\pi_1(B_a')$.  Thus, to estimate the $2^{-2l}$-cover number of $\pi_2\left( \bigcup_{a\in A}B_a' \right)$, we only need to estimate the following
\[
N_{2^{-2l}}\left( \bigcup_{a\in A}(B_a+ \pi_2(x_a) -z_i) \right).
\]
For this, we shall use Theorem \ref{thm:entropy-increase-1}.   
First,  let us recall that we have $\mu_{\pi_1^{-1}(z_i)} (L')>1-\epsilon_3$ (recall  \eqref{eq:thm-proj-cp-distribution-1:i}) and the 
measure $\mu_{\pi_1^{-1}(z_i)}$ is $(n,l_0,\epsilon_0)$-dyadic spreading.
The set $A'=\{x_a:a\in A\}\subset L'$ and $d_{\rm H}(A',L')\le c\cdot 2^{-l}$.
The set $B$ satisfies \eqref{eq:thm-proj-cp-distribution-1:11}.  Since we initially assume $\epsilon_0,\epsilon_1,\epsilon_2$ are small enough, so that we have $\beta_0+2(\epsilon_2+ \epsilon_1) \le \beta_0+(1-\beta_0)/2<1$ and  $1-\epsilon_0>\frac{1}{2}$ and $1-2\epsilon_0>\beta_0+2(\epsilon_2+ \epsilon_1) $.  Thus we may effectively apply Theorem \ref{thm:entropy-increase-1} to conclude that  there exists $\delta>0$, only depends on $l_0$ and $\beta$ such that 
\[
N_{2^{-2l}}\left( \bigcup_{a\in A}(B_a+ \pi_2(x_a) -z_i) \right)\ge \left(N_{2^{-2l}}(B)\right)^{1+\delta}\ge 2^{l(\beta_0+\delta/2)}
\]
provided $\epsilon_1, \epsilon_2,  \epsilon_3$ are assumed small enough.
Thus we have shown that 
\begin{equation}\label{eq:thm-proj-cp-distribution-1:19}
N_{2^{-2l}}\left( \pi_2\left( \bigcup_{a\in A}B_a' \right) \right) \ge c\cdot 2^{l(\beta_0+\delta/2)}
\end{equation}
for some absolute constant $c>0$.  Note that for each $a\in A$,  we have $\pi_1(B_a')\subset B(z_i,2^{-l+1})$.  Hence in particular \eqref{eq:thm-proj-cp-distribution-1:19} implies that 
\begin{equation*}\label{eq:thm-proj-cp-distribution-1:20}
N_{2^{-2l}}\left( \pi_2\left( \pi_1^{-1}(B(z_i,2^{-l+1})) \right) \right) \ge c\cdot 2^{l(\beta_0+\delta/2)}.
\end{equation*}
This is what we aimed to prove.
\end{proof}

\section{Miscellaneous Proofs}\label{section:remainning proofs}

\subsection{Proof of Theorem  \ref{thm:main-proj-upper-unif-entropy-dim}}

In the following we shall prove Theorem  \ref{thm:main-proj-upper-unif-entropy-dim} using Theorem   \ref{thm:proj-thm-cp-distribution:-2}.   Before giving the proof, we need some preparations.

Let us first prove the following properties for measures with upper uniform entropy dimension $\alpha$.
\begin{lemma}\label{lemma:unif-dim-meas-generate-unif-dim-CP-dist}
Let $\mu\in \cP(\R^d)$. Suppose that $\mu$ has upper uniform entropy dimension $\alpha$. Then for any $\epsilon>0$, there exists Borel set $A\subset \R^d$ with $\mu(A)>1-\epsilon$ such that for each $x\in A$, there exist $(m_k)_k\subset \N$ and CP-distribution $Q_x$ such that 
\[
\frac{1}{m_k}\sum_{n=1}^{m_k} \delta_{\{ \mu^{\D_n(x)} \}}  \rightharpoonup  Q_x \  \textrm{ as } k\to \infty
\]
and
\[
Q_x\left(\eta: |\dim \eta-\alpha|<\epsilon\right)>1-\epsilon.
\]
\end{lemma}
\begin{proof}
By definition of upper uniform entropy dimension, for any $\epsilon>0$, there exist arbitrarily large  $l\in \N$ such that 
\begin{equation}\label{eq:proof:lemma:unif-dim-meas-generate-unif-dim-CP-dist:1}
\limsup_{n\to\infty}\mu\left( x:\frac{1}{n}\left| \left\{ 1\le k\le n: \left|H(\mu^{\D_k(x)},\D_l)-l\alpha\right| \le l \epsilon \right\} \right| \ge n(1-\epsilon) \right) \ge 1-\epsilon.
\end{equation}
Let 
\[
A_n=\left\{ x:\frac{1}{n}\left| \left\{ 1\le k\le n: \left|H(\mu^{\D_k(x)},\D_l)-l\alpha\right| \le l \epsilon \right\} \right| \ge n(1-\epsilon)  \right\}.
\]
Then there exists subsequence $(n_k)_k\subset \N$ such that $\mu(A_{n_k})>1-2\epsilon$ for all  $k\ge 1$.  Let 
\[
A'=\limsup_{k\to\infty} A_{n_k}.
\]
Thus we have 
\[
\mu(A')=\mu\left( \cap_{N\ge 1}\cup_{k\ge N} A_{n_k} \right)\ge 1-2\epsilon.
\]
For each $x\in A'$, there exist infinite sequence $(m_k(x))_k\subset (n_k)_k$ such that  $x\in A_{m_k(x)}$ for all $k\ge 1$.  For each $x\in A'$, let us fix a distribution $Q_x$, which is limit point  of the sequence 
\begin{equation}\label{eq: proof:lemma:unif-dim-meas-generate-unif-dim-CP-dist:2}
\left( \frac{1}{m_k(x)}\sum_{n=1}^{m_k(x)} \delta_{\{\mu^{\D_n(x)}\}} \right)_{k\ge 1}.
\end{equation}
By \cite[Section 7.5, Theorem 7.1]{HS2012}, we know that for $\mu$-a.e.  $x\in A$, $Q_x$ is a CP-distribution.  Note that the dimension of a CP-distribution $Q$ is given by the following formula (see \cite[Proposition 6.26]{Hochman-lecture-notes2014})
\begin{equation}\label{eq: proof:lemma:unif-dim-meas-generate-unif-dim-CP-dist:3}
\dim Q= \int \frac{1}{l} H(\eta,\D_l)dQ(\eta) \ \  \textrm{for any } l\ge 1.
\end{equation}
Recall that the dimension of $Q$ is also defined as 
\begin{equation}\label{eq: proof:lemma:unif-dim-meas-generate-unif-dim-CP-dist:3.1}
\dim Q = \int \dim\eta dQ(\eta).
\end{equation}
Since $Q_x$ is a limit point of \eqref{eq: proof:lemma:unif-dim-meas-generate-unif-dim-CP-dist:2}, up to taking a subsequence of $(m_k(x))_k$, we may assume 
\begin{equation}\label{eq: proof:lemma:unif-dim-meas-generate-unif-dim-CP-dist:4}
\frac{1}{m_k(x)}\sum_{n=1}^{m_k(x)} \delta_{\{\mu^{\D_n(x)}\}}  \rightharpoonup  Q_x \textrm{ as } k\to \infty.
\end{equation}
By definition of $(A_{m_k(x)})_k$, we have for each $k\ge 1$, 
\begin{equation}\label{eq: proof:lemma:unif-dim-meas-generate-unif-dim-CP-dist:5}
\frac{1}{m_k(x)}\left| \left\{ 1\le k\le m_k(x): \left|H(\mu^{\D_k(x)},\D_l)-l\alpha\right| \le l \epsilon \right\} \right| \ge m_k(x)(1-\epsilon)
\end{equation}
Define $f:\cP(\R^d)\to \R$ by setting
\[
f(\eta)=\int_{[0,1]^d}H\left(\eta(\cdot+t),\D_l\right)dt.
\]
Since for $\mathcal{L}^d$-a.e.  $t\in [0,1]^d$, the measure $\eta(\cdot+t)$ gives zero mass to the boundaries of elements of $\D_l$, the function $f$ is continuous.  This is the advantage of considering $f$ instead of the function $\eta\mapsto H(\eta,\D_l)$.  By Lemma \ref{lemma:almost-continuity-entropy:1} (2), we know that 
\begin{equation}\label{eq: proof:lemma:unif-dim-meas-generate-unif-dim-CP-dist:6}
\left|f(\eta)-H(\eta,\D_l)\right| = O_d(1). 
\end{equation}
By the weak convergence \eqref{eq: proof:lemma:unif-dim-meas-generate-unif-dim-CP-dist:4},  we get 
\[
\lim_{k\to\infty}\frac{1}{m_k(x)}\sum_{n=1}^{m_k(x)}f(\mu^{\D_n(x)})  = \int f(\eta)  dQ_x(\eta).
\]
By \eqref{eq: proof:lemma:unif-dim-meas-generate-unif-dim-CP-dist:5} and \eqref{eq: proof:lemma:unif-dim-meas-generate-unif-dim-CP-dist:6},  we have
\[
\frac{1}{m_k(x)}\sum_{n=1}^{m_k(x)}f(\mu^{\D_n(x)})  \ge \alpha-O_d(\frac{1}{l})-o_\epsilon(1).
\]
It follows that we have
\[
\int \frac{1}{l} H(\eta,\D_l)dQ_x(\eta) \ge  \alpha-O_d(\frac{1}{l})-o_\epsilon(1).
\]
Combining this with \eqref{eq: proof:lemma:unif-dim-meas-generate-unif-dim-CP-dist:3} and \eqref{eq: proof:lemma:unif-dim-meas-generate-unif-dim-CP-dist:3.1},   we deduce that there exists $\epsilon'=O_d(\frac{1}{l})+o_\epsilon(1)$ such that 
\[
Q_x\left(\eta: |\dim \eta-\alpha|<\epsilon'\right)>1-\epsilon'.
\]
Putting together the above arguments, we thus obtain the desired conclusion.
\end{proof}

\begin{proof}[Proof of Theorem \ref{thm:main-proj-upper-unif-entropy-dim}.]
Let $\mu$ be a measure satisfying the assumptions of Theorem \ref{thm:main-proj-upper-unif-entropy-dim}.  Suppose, towards a contradiction that the set \eqref{eq:thm:main-proj-upper-unif-entropy-dim:1} is uncountable.  One may check that the set \eqref{eq:thm:main-proj-upper-unif-entropy-dim:1} is a Borel set.  Then for small enough $\delta>0$, the following set is still uncountable and it supports a non-atomic Borel probability measure $\nu$. 
\[
E=\{\pi\in G(2,1): \dimP \pi\mu<\min (1,\alpha)-\delta\}.
\]
Recall that we have the following  characterization for Packing dimension of measures (see Section \ref{section:notation}, \eqref{eq:chara-packing-dim-meas:1}): for $\eta\in \cP(\R^d)$,  
\begin{equation}\label{eq: Packing-dim-characterization:1}
\dimP \eta = \ess-inf_{x\sim \eta} \overline{D}(\eta,x).
\end{equation}
It follows that for each $\pi\in E$, there exists $F_\pi\subset \R^2$  with $\mu(F_\pi)>0$ and
\begin{equation}\label{eq:proof:thm-proj-upp-uni-dim-by-proj-cp-dist: 2}
\overline{D}(\pi\mu,\pi(x)) \le \min(1,\alpha)-\delta \textrm{ for } x\in F_\pi.
\end{equation} 
Let us recall a useful fact from \cite[Theorem 5.4]{HS2012}: there exists absolute constant $c>0$, depending only on $\R^2$, such that for any $\eta\in \cP(\R^2)$ and $\pi\in G(2,1)$,   we have
\begin{equation}\label{eq:proof:thm-proj-upp-uni-dim-by-proj-cp-dist: 3}
\limsup_{n\to\infty} \frac{1}{n}\sum_{k=1}^n \frac{1}{l}H(\pi(\eta^{\D_k(x)}),\D_l) \le \overline{D}(\pi\eta,\pi(x)) +\frac{c}{l} \textrm{ for all } l\ge 1.  
\end{equation} 
From \eqref{eq:proof:thm-proj-upp-uni-dim-by-proj-cp-dist: 2} and \eqref{eq:proof:thm-proj-upp-uni-dim-by-proj-cp-dist: 3},  it follows that, assuming $l_0$ is large enough so that $c/l_0\le \delta/2$,  for each $\pi\in E$, there exists $F_\pi$ with $\mu(F_\pi)>0$ and for each $x\in F_\pi$,
\begin{equation}\label{eq:proof:thm-proj-upp-uni-dim-by-proj-cp-dist: 4}
\limsup_{n\to\infty} \frac{1}{n}\sum_{k=1}^n \frac{1}{l}H(\pi(\mu^{\D_k(x)}),\D_l) \le \min(1,\alpha)-\frac{\delta}{2} \  \textrm{ for all } l\ge l_0.  
\end{equation} 
By Fubini's theorem and Egorov's theorem,  there exist $\epsilon_0>0$ and $F\subset \R^2$ with $\mu(F)>\epsilon_0$ such that for each $x\in F$ there exists $E^0_x\subset E$ with $\nu(E^0_x)>\epsilon_0$ and for all $\pi\in E^0_x$,  the inequalities \eqref{eq:proof:thm-proj-upp-uni-dim-by-proj-cp-dist: 4} hold. 

Now, let us fix $\epsilon_1$ with $0<\epsilon_1<\epsilon_0/2$.  Applying Lemma \ref{lemma:unif-dim-meas-generate-unif-dim-CP-dist},  we obtain a set  $A\subset \R^2$ with $\mu(A)>1-\epsilon_1$ such that for each $x\in A$, there exist $(m_k(x))_k\subset \N$ and CP-distribution $Q_x$ such that 
\begin{equation}\label{eq:proof:thm-proj-upp-uni-dim-by-proj-cp-dist: 5}
\frac{1}{m_k(x)}\sum_{n=1}^{m_k(x)} \delta_{\{ \mu^{\D_n(x)} \}}  \rightharpoonup  Q_x \  \textrm{ as } k\to \infty
\end{equation} 
and
\begin{equation}\label{eq:proof:thm-proj-upp-uni-dim-by-proj-cp-dist: 6}
Q_x\left(\eta: |\dim \eta-\alpha|<\epsilon_1\right)>1-\epsilon_1.
\end{equation} 
Let us fix such  $x\in A$,  $(m_k(x))_k\subset \N$ and $Q_x$.   We know that (see \cite[theorem 1.22]{Hochman2010-fractal-distributions}) for any $\pi\in G(2,1)$,  for $Q_x$-a.e.  $\eta$,   the measure $\pi\eta$ is exact dimensional. Thus by Egorov's theorem, there exist $A'\subset A$ and $l_1\in \N$ with $\mu(A')\ge \mu(A)-\epsilon_1$ such that for any $x\in A'$, there exists $E^1_x\subset G(2,1)$ with $\nu(E^1_x)>1-\epsilon_1$ and  for each  $\pi\in E^1_x$ there exists $D_x^1\subset \cP(\R^2)$ so that $Q_x(D_x^1)>1-\epsilon_1$ and for each $\eta\in D^1_x$,  we have
\begin{equation}\label{eq:proof:thm-proj-upp-uni-dim-by-proj-cp-dist: 7}
\left| \dim \pi\eta -\frac{1}{l}H(\pi\eta,\D_l) \right| <\epsilon_1 \ \textrm{ for } \eta\in D_x^1,  l\ge l_1.
\end{equation}
Define  $f:\cP(\R^2)\to \R$ by setting
\[
f(\eta)=\int_{[0,1]^2}\frac{H\left(\pi(\eta(\cdot+t)),\D_l\right)}{l}dt.
\]
Then the  function $f$ is continuous.  By Lemma \ref{lemma:almost-continuity-entropy:1} (2), we know that for each $\eta\in \cP(\R^2)$, 
\begin{equation}\label{eq:proof:thm-proj-upp-uni-dim-by-proj-cp-dist: 8}
\left|f(\eta)-\frac{1}{l}H(\pi\eta,\D_l)\right| = O\left(\frac{1}{l}\right). 
\end{equation}
By the weak convergence \eqref{eq:proof:thm-proj-upp-uni-dim-by-proj-cp-dist: 5}, we have 
\begin{equation}\label{eq:proof:thm-proj-upp-uni-dim-by-proj-cp-dist: 9}
\lim_{k\to\infty}\frac{1}{m_k(x)}\sum_{n=1}^{m_k(x)}   f(\mu^{\D_n(x)})  =  \int  f(\eta) d Q_x(\eta).
\end{equation}
By \eqref{eq:proof:thm-proj-upp-uni-dim-by-proj-cp-dist: 8}, we get 
\begin{equation}\label{eq:proof:thm-proj-upp-uni-dim-by-proj-cp-dist: 10}
\liminf_{k\to\infty}\frac{1}{m_k(x)}\sum_{n=1}^{m_k(x)}  \frac{1}{l}H(\pi(\mu^{\D_n(x)}),\D_l)   \ge   \int  \frac{1}{l}H(\pi\eta,\D_l) d \ Q_x(\eta) - O\left(\frac{2}{l}\right).
\end{equation}
Let $A''=A'\cap F$. Then 
\[\mu(A'')>\epsilon_0-\epsilon_1-\epsilon_1>0.\] 
In the following, we fix $l\ge \max(l_0,l_1)$.
By \eqref{eq:proof:thm-proj-upp-uni-dim-by-proj-cp-dist: 4} and \eqref{eq:proof:thm-proj-upp-uni-dim-by-proj-cp-dist: 10},  we get that for each $\pi\in E_x^0$, 
\begin{equation}\label{eq:proof:thm-proj-upp-uni-dim-by-proj-cp-dist: 11}
\int  \frac{1}{l}H(\pi\eta,\D_l) \ d Q_x(\eta) \le \min(1,\alpha)-\frac{\delta}{2} + O\left(\frac{2}{l}\right).
\end{equation}
By \eqref{eq:proof:thm-proj-upp-uni-dim-by-proj-cp-dist: 7},  we know that for any $\pi\in E_x^1$,  we have
\begin{equation}\label{eq:proof:thm-proj-upp-uni-dim-by-proj-cp-dist: 12}
\int  \frac{1}{l}H(\pi\eta,\D_l) \ d Q_x(\eta) \ge \int \dim \pi\eta dQ_x(\eta)- o_{\epsilon_1}(1).
\end{equation}
Combining \eqref{eq:proof:thm-proj-upp-uni-dim-by-proj-cp-dist: 11} and \eqref{eq:proof:thm-proj-upp-uni-dim-by-proj-cp-dist: 12},  we obtain that for each $\pi\in  E_x^0\cap  E_x^1$,  we have 
\[
 \int \dim \pi\eta \ dQ_x(\eta)\le  \min(1,\alpha)-\frac{\delta}{2} + O\left(\frac{2}{l}\right) +o_{\epsilon_1}(1).
\]
We may assume that $l_0, l_1$ has been chosen large enough and $\epsilon_1$ small enough so that $O\left(\frac{2}{l}\right) +o_{\epsilon_1}(1)<\delta/4$.  Thus we have
\begin{equation}\label{eq:proof:thm-proj-upp-uni-dim-by-proj-cp-dist: 13}
 \int \dim \pi\eta \  dQ_x(\eta)\le  \min(1,\alpha)-\frac{\delta}{4} \ \textrm{ for } \pi\in E_x^0\cap E_x^1.
\end{equation}
Note that we have 
\begin{equation}\label{eq:proof:thm-proj-upp-uni-dim-by-proj-cp-dist: 14}
 \nu(E_x^0\cap E_x^1)\ge \epsilon_0-\epsilon_1.
\end{equation}
Recall that we have $0<\epsilon_1<\epsilon_0/2$.
Let $Q_x=\int Q_x^\omega \ d\tau(\omega)$ be the ergodic decomposition of $Q_x$. Then we have 
\begin{equation}\label{eq:proof:thm-proj-upp-uni-dim-by-proj-cp-dist: 15}
\int \dim \pi\eta \ dQ_x(\eta) = \int \left(\int \dim \pi\eta \ dQ_x^\omega(\eta)\right)  d\tau(\omega).
\end{equation}
we also have 
\begin{equation}\label{eq:proof:thm-proj-upp-uni-dim-by-proj-cp-dist: 16}
\int \dim  \eta \ dQ_x(\eta) = \int \left(\int \dim  \eta \ dQ_x^\omega(\eta)\right)  d\tau(\omega).
\end{equation}
Combining \eqref{eq:proof:thm-proj-upp-uni-dim-by-proj-cp-dist: 6},  \eqref{eq:proof:thm-proj-upp-uni-dim-by-proj-cp-dist: 13}, \eqref{eq:proof:thm-proj-upp-uni-dim-by-proj-cp-dist: 14}, \eqref{eq:proof:thm-proj-upp-uni-dim-by-proj-cp-dist: 15} and \eqref{eq:proof:thm-proj-upp-uni-dim-by-proj-cp-dist: 16},  we obtain that,  assuming $\epsilon_1$ is small enough in terms of $\delta$,  there exist $G$ with $\tau(G)>0$ and $\delta'>0$ such that for each $\omega\in G$, there exist $E_2^\omega\subset G(2,1)$ with $\nu(E_2^\omega)>\delta'$, such that 
\[
\dim \pi Q_x^\omega  = \int \dim \pi \eta \ dQ_x^\omega(\eta) \le \min(1,\dim Q_x^\omega)-\delta' \textrm{ for } \omega\in G,  \pi\in E_2^\omega.
\]
Since $\nu$ is non-atomic, the set $E_2^\omega$ is uncountable,  thus we obtain a contradiction to Theorem \ref{thm:proj-thm-cp-distribution:-2}.
\end{proof}

\subsection{Proof of Theorem \ref{thm:proj-Assouad:-1}}\label{subsection:Proj-Assouad}

Before giving the proof, we need some preparations.  Recall that for $D\in \D_k(\R^2)$,  $S_{D}$ is the unique orientation-preserving homothety sending $D$ to $[0,1)^2$.  Let $F\subset \R^2$.  We say $F'$ is s micro set of $F$ if there exist a sequence $(k_j)_j\subset \N$ and $D_j\in \D_{k_j}(\R^2)$ such that 
\[d_{\rm H}(F',D_j\cap F)\to  0 \ \textrm{ as } j\to\infty,\]
where $d_{\rm H}$ stands for the Hausdorff distance. 

We will use the following well known facts about Assouad dimension.
\begin{lemma}[\cite{Furstenber2008,Fraser-book}]\label{lemma:proof-proj-Assouad:1}
\begin{itemize}
\item[(1)] Let $F\subset \R^d$.  For any micro set $F'$ of $F$, we have $\dim_{\rm A}F\ge \dim_{\rm A}F'$.
\item[(2)] Let $F\subset \R^d$. There is a micros set $F'$ of $F$ such that $F'$ supports a probability measure $\mu$ and for $\mu$-a.e.  $x$, the scenery sequence of $\mu$ at $x$ generates an ergodic CP-distribution $Q$ with $\dim Q=\dim_{\rm A}F$.
\end{itemize}
\end{lemma}

Now we are ready to prove Theorem \ref{thm:proj-Assouad:-1}.
\begin{proof}[Proof of Theorem \ref{thm:proj-Assouad:-1}]
By Lemma \ref{lemma:proof-proj-Assouad:1} (2), there exists a micro set $F'$ of $F$ such that one can find $\mu\in \cP(F')$ satisfying for $\mu$-a.e.  $x$,  
\[\lim_{n\to\infty} \frac{1}{n}\sum_{k=1}^n\delta_{\mu^{\D_k(x)}}  \ = \  Q \]
for some ergodic CP-distribution $Q$ with $\dim Q\ge \dim_{\rm A} F$. By \cite[Theorem 8.1]{HS2012},  we have for any $\pi\in G(2,1)$, 
\begin{equation}\label{eq:pf-thm-assouad-proj:1}
\dimH \pi\mu\ge \dim \pi Q.
\end{equation}
By Theorem \ref{thm:proj-thm-cp-distribution:-2},  we have 
\begin{equation}\label{eq:pf-thm-assouad-proj:2}
\dim \pi Q = \min(1,\dim Q)
\end{equation}
for all $\pi\in G(2,1)\setminus E$ where $E$ is at most countable.  By Lemma \ref{lemma:proof-proj-Assouad:1} (1),  we have for each $\pi\in G(2,1)$,
\begin{equation}\label{eq:pf-thm-assouad-proj:3}
\dim_{\rm A}\pi(F) \ge \dim_{\rm A}\pi(F').
\end{equation}
Since $\dim_{\rm A}\pi(F')\ge \dimH \pi\mu$,   by combining \eqref{eq:pf-thm-assouad-proj:1}, \eqref{eq:pf-thm-assouad-proj:2} and \eqref{eq:pf-thm-assouad-proj:3},   we obtain the desired conclusion

\end{proof} 
\begin{remark}
Theorem \ref{thm:proj-Assouad:-1} is sharp: there exists a self-similar set $K\subset \R^2$ satisfying the open set condition (thus $K$ is Alfhors-David regular) such that $\{\pi\in G(2,1): \dim_{\rm A}\pi K<\min(1\dim_{\rm A}K)\}$ is dense in $G(2,1)$ (hence at least countable). Here is such an example: consider the $\frac{1}{4}$-four-corner Cantor set $K\subset [0,1]^2$ generated by the IFS 
\[
\F=\l\{(x,y)\mapsto \l(\frac{x}{4},\frac{x}{4}\r)+\l(\frac{a}{4},\frac{b}{4}\r): (a,b)\in \{(0,0),(3,0),(0,3),(3,3)\}\r\}=\{f_i\}_{i\in\Lambda}.
\]
Then $K$ satisfies the open set condition and $\dimH K=1$.  Thus $K$ is Alfhors-David regular and $\dim_{\rm A}K=\dim_{\rm H}K=1$.  It is well known that the set $K$ is purely 1-unrectifiable.  By a classical projection theorem  of Besicovitch we know that $\mathcal{L}(\pi K)=0$ for $\mathcal{L}$-a.e.  $\pi\in G(2,1)$.  

Let us fix any $\pi\in G(2,1)$ such that $\mathcal{L}(\pi K)=0$.   We will show that for any $\epsilon>0$, there exists $\pi'\in G(2,1)$ with $0<|\pi-\pi'|<\epsilon$ such that
\[
\dim_{\rm A}\pi'K<\dim_{\rm A} K=1.
\]
Since $\mathcal{L}(\pi K)=0$, for any (large) $M\in \N$, there exist large enough $n\in\N$ such that there exist distinct $I_1,\cdots, I_M\in \Lambda^n$ satisfying 
\[
\pi (f_{I_i}([0,1]^2))\bigcap \pi (f_{I_j}([0,1]^2))\neq \emptyset \ \textrm{ for all } 1\le i,j\le M.
\]
Observe that we have
\[
\dist\l(f_{I_i}([0,1]^2),f_{I_j}([0,1]^2)\r)\ge \frac{1}{4^n}\  \textrm{ for all } i\neq j.
\]
Thus there must exist $1\le i,j\le M$ with $i\neq j$ such that 
\[
\dist\l(f_{I_i}([0,1]^2),f_{I_j}([0,1]^2)\r)\ge \frac{1}{4^n}\cdot \frac{M}{2}.
\]
Let $\pi'\in G(2,1)$ be such that the projected sets $\pi'(f_{I_i}([0,1]^2))$ and $\pi'(f_{I_j}([0,1]^2))$ exactly coincide, that is, $\pi'(f_{I_i}([0,1]^2))=\pi'(f_{I_j}([0,1]^2))$.
Then $\pi'$ is rational  and $|\pi-\pi'|\le o_M(1)$.  Moreover, we have 
\[
\dimH \pi' K<\dimH K=1.
\]
Since $\pi'$ is rational, we know that the projected IFS $\F_{\pi'}=\{\pi' f_i\}_{i\in \Lambda}$ satisfies the weak separation condition (see e.g.  \cite{Ruiz}).  By a result of Fraser et al \cite{FHOR}, we know that the attractor of $\F_{\pi'}$, which is $\pi'(K)$, is Alfhors-David regular.  In particular, we have 
\[
\dim_{\rm A}\pi' K=\dim_{\rm H}\pi' K.
\]
Since we have seen that $\dim_{\rm H}\pi' K<1$, we have $\dim_{\rm A}\pi' K<\dim_{\rm A}K=1$. 
Thus we have proved that for any $\pi\in G(2,1)$ with $\mathcal{L}(\pi  K)=0$ and any $\epsilon>0$,  there exists $\pi' \in S^1$ with $0<|\pi-\pi'|<\epsilon$ such that
\[
\dim_{\rm A}\pi' K<\dim_{\rm A} K=1.
\]
Since $\mathcal{L}(\pi  K)=0$ for $\mathcal{L}$-a.e. $\pi\in G(2,1)$,  the set $E_{\rm A}(K)$ must be dense in $G(2,1)$.

\end{remark}

\medskip

\noindent{\textbf{Acknowledgements.}} I am grateful to Mike Hochman for his encouragement and valuable suggestions, which have significantly improved this paper. I also thank Pablo Shmerkin for his encouragement and Amir Algom for his comments on an earlier version.


\end{document}